\documentclass[11pt]{amsart}
\usepackage{cite}
\usepackage{amsmath, amssymb, amsthm}
\usepackage{booktabs, url}
\usepackage[table]{xcolor}

\newcommand{\thm}[2][]{\def\firstargtemp{#1}%
    Theorem\ifx\firstargtemp\empty\else s~\ref{thm:#1} and\fi~\ref{thm:#2}}

\newcommand{\lem}[1]{Lemma~\ref{lem:#1}}
\newcommand{\rmk}[1]{Remark~\ref{rmk:#1}}
\newcommand{\corr}[1]{Corollary~\ref{corr:#1}}
\newcommand{\claim}[1]{Claim~\ref{#1}}
\newcommand{\cnj}[1]{Conjecture~\ref{cnj:#1}}
\newcommand{\tbl}[2][]{\def\firstargtemp{#1}%
    Table\ifx\firstargtemp\empty\else s~\ref{tbl:#1} and\fi~\ref{tbl:#2}}

\def\etal{et~al.\ }
\newcommand\fH[1]{\sbox0{#1}\dimen0=\ht0 \advance\dimen0 -1ex
  \sbox2{\'{}}\sbox2{\raise\dimen0\box2}%
  {\ooalign{\hidewidth\kern.1em\copy2\kern-.5\wd2\box2\hidewidth\cr\box0\crcr}}}

\renewcommand\labelenumi{(\roman{enumi})}
\renewcommand\theenumi\labelenumi

\newtheorem{lemma}{Lemma}[section]
\newtheorem{remark}[lemma]{Remark}

\newtheorem{theorem}[lemma]{Theorem}
\newtheorem{corollary}[lemma]{Corollary}
\newtheorem{conjecture}[lemma]{Conjecture}
\newtheorem*{definition}{Definition}
\newtheorem{clam}{Claim}[lemma]

\DeclareMathOperator{\Sym}{Sym}
\DeclareMathOperator{\Aut}{Aut}
\DeclareMathOperator{\vspan}{span}

\newcommand\sym[1]{{\rm S}_{#1}}
\newcommand\alt[1]{{\rm A}_{#1}}
\newcommand\dih[1]{{\rm D}_{#1}}
\newcommand\MT[1]{{\rm M}_{#1}}
\newcommand\cyc[1]{{\rm C}_{#1}}
\newcommand\suz[1]{{}^2 {\rm B}_2 (#1)}
\newcommand\ree[1]{{}^2 {\rm G}_2 (#1)}

\newcommand\psl[2]{{\rm PSL}_{#1} (#2)}
\newcommand\pgl[2]{{\rm PGL}_{#1} (#2)}
\newcommand\lins[2]{{\rm SL}_{#1} (#2)}
\newcommand\ling[2]{{\rm GL}_{#1} (#2)}
\newcommand\psigmal[2]{{\rm P\Sigma L}_{#1} (#2)}
\newcommand\pgammal[2]{{\rm P\Gamma L}_{#1} (#2)}
\newcommand\gammal[2]{{\rm \Gamma L}_{#1} (#2)}
\newcommand\psu[2]{{\rm PSU}_{#1} (#2)}
\newcommand\psigmau[2]{{\rm P\Sigma U}_{#1} (#2)}
\newcommand\symp[2]{{\rm Sp}_{#1} (#2)}
\newcommand\xlie[3]{{\rm #1}_{#2} (#3)}
\newcommand\vsp[2]{V_{#1} (#2)}

\newcommand\ind[1]{{\rm I}_{#1}}

\newcommand{\saxg}[1][G]{\Sigma (#1)}
\newcommand{\saxh}[1][G]{\mathcal{H} (#1)}
\def\GosNum{Gossip Number}
\NewDocumentCommand{\gosnum}{ O{} O{th} }{\def\firstargtemp{#1}%
    \ifx\firstargtemp\empty\else $#1^{\rm #2}$ \fi~gossip number}
\newcommand\komp[1]{{\mathcal K} (#1)}

\author{Melissa Lee \and Anthony Pisani}
\address{School of Mathematics, Monash University, Clayton VIC 3800, Australia}
\email{\{melissa.lee, anthony.pisani\}@monash.edu}
\date{\today}
\title{The Saxl hypergraph of a permutation group}

\begin{document}
    \maketitle

    \begin{abstract}
        Given a permutation group $G \le \Sym{(\Omega)}$, a subset $B$
        of $\Omega$ is said to be a base if its pointwise stabiliser
        in $G$ is trivial, and the base size $b(G)$ is the minimum size
        of a base. In the notable case $b(G) = 2$, Burness and Giudici define
        the Saxl graph of $G$ to be the graph on $\Omega$ with bases of size 2
        as edges. Later work of Freedman \etal extends this notion
        to any group for which $b(G) \ge 2$, taking the pairs of points
        contained in bases of size $b(G)$ for edges. We study
        an alternative generalisation, the Saxl hypergraph,
        where bases of size $b(G)$ are themselves the edges. In particular,
        we consider groups with complete Saxl hypergraphs, primitive groups
        whose Saxl hypergraphs have flag-spanning tours, and
        appropriate generalisations of Burness and
        Giudici's Common Neighbour Conjecture.
    \end{abstract}

    \section{\label{sec:intro}Introduction}

    Recall that, given a permutation group $G \le \Sym{(\Omega)}$, a subset $B$
    of $\Omega$ is said to be a \textsl{base} for $G$ if its pointwise stabiliser
    therein is trivial. The \textsl{base size} $b(G)$ of $G$ is the minimum size
    of a base. These concepts have a lengthy history dating back
    to the $19^{\rm th}$ century, yet remain a vibrant area of
    contemporary mathematical research with broad relevance in algebra and
    combinatorics; we refer the reader to the survey articles
    \cite{prodbase,liesha} and \cite[\S5]{fpr_survey}
    for a more detailed overview, and to \cite{akperm,gls} as examples of
    applications. \\

    Motivated by the use of bases in computational group theory, it is of particular interest to identify finite groups with small base sizes.
    Although only weak upper bounds on base sizes are possible in general ---
    Pyber \cite{pyber_asym} shows that almost all groups of degree $n$ have base size
    at least $cn$ for some fixed $c > 0$ --- recent work has revealed
    that many families of important groups are much better-behaved,
    especially primitive groups. For instance,
    significant conjectures made by Babai, Cameron, Kantor, and Pyber
    in the 1990s on the asymptotic properties of $b(G)$ in various scenarios
    have now been proved with explicit, strong upper limits. These include
    the surprising theorem of Burness \etal\cite{basecls, basesym, basespor, basesim} that base sizes of
    almost simple primitive groups are bounded above (barring
    those in \textsl{standard} actions), not just by a fixed constant as Cameron and Kantor
    \cite{camkant} conjectured, but by the very manageable limit of 7; in fact,
    with the sole exception of $\MT{24}$ acting on 24 points, by 6. \\

    At a near-extreme, a notable class of small-base groups
    are the finite primitive groups of base size 2. (The case $b(G) = 1$
    corresponds to prime cyclic groups.) A classification of such groups,
    still incomplete today, forms the aim of a research project initiated
    by Jan Saxl in the 1990s. Further recent interest stems from a graph defined
    for such groups by Burness and Giudici \cite{burngiu}:

    \begin{definition}
        The {\rm Saxl graph} $\saxg$ of a group $G\leq \mathrm{Sym}(\Omega)$ is
        the graph on vertex set $\Omega$ with edge set
        $\{ \alpha \beta \, : \, \{ \alpha, \beta \} \textrm{ is a base for } G \}$.
    \end{definition}

    Although a number of properties of $\saxg$ are considered
    in this initial paper and occasionally in the later literature
    \cite{chen+huang,saxlsol,prodsoltab}, subsequent attention
    has clustered around a striking conjecture {\cite[p.2]{burngiu}} they put forth concerning
    the graph's diameter.

    \begin{conjecture}[Common Neighbour Conjecture (CNC)]
        \label{cnj:cnc}
        Let $G\leq \mathrm{Sym}(\Omega)$ be a primitive permutation group with $b(G) = 2$.
        Any two vertices of $\saxg$ share a common neighbour.
        As an immediate corollary, ${\rm diam} (\saxg) \le 2$.
    \end{conjecture}

    Burness and Giudici themselves verify the preceding claim when
    $| \Omega | \le 4095$, $G$ is a sufficiently large group of
    twisted wreath or diagonal type in the O'Nan--Scott classification
    of primitive groups (see \S\ref{onan+scott}), or $G$ is one of a number
    of almost simple permutation groups with sporadic simple or
    alternating socle. Extending work of Chen and Du \cite{chen+du},
    Burness and Huang \cite{saxlsol} offer a proof in the cases
    where $G$ is almost simple simple with socle $\psl{2}{q}$ or has
    a soluble point stabiliser. In further work \cite{prodsoltab}, they
    also check the conjecture for groups $L \wr P$ of product type. Meanwhile,
Lee and Popiel \cite{lee+pop} confirm
    the conjecture for most affine groups with sporadic point stabilisers. \\

    More recently, efforts have been undertaken to develop
    an interesting analogue of the Saxl graph for groups of arbitrary base size. 
    In this spirit, Freedman \etal\cite{fhlk} introduce and study
    the following graph:

    \begin{definition}
        The {\rm generalised Saxl graph} $\saxg$ of a group $G\leq\mathrm{Sym}(\Omega)$ is
        the graph on vertex set $\Omega$ with edges between vertices $\alpha,\beta \in \Omega$ if and only if $\alpha,\beta$ lie in a base for $G$ of size $b(G)$.
    \end{definition}
    \begin{remark}
        Note that the generalised Saxl graph and Saxl graph coincide
        if $b(G) = 2$. Furthermore, the graphs have no edges when $b(G) < 2$ or
        $b(G) > 2$, respectively, and so are never discussed in these contexts,
        obviating any potential confusion from overloading the symbol $\saxg$.
    \end{remark}

    Their analysis, among other things, suggests that \cnj{cnc} continues
    to hold in its natural generalisation to generalised Saxl graphs
    of groups for which $b(G) \ge 2$. \\

    This paper seeks to investigate an alternative generalisation of
    Saxl graphs. Recall that a \textsl{hypergraph} may be defined as a pair
    $(V, E)$, where $V$ is the set of vertices and $E \subseteq \mathcal{P} (V)$
    the set of arbitrarily-sized \textsl{(hyper)edges}.

    \begin{definition}
        The {\rm Saxl hypergraph} $\saxh$ of a group $G \le \Sym{(\Omega)}$ is
        the hypergraph on vertex set $\Omega$ with edge set consisting of all bases for $G$ of size $b(G)$.
    \end{definition}

    Like the generalised Saxl graph, the Saxl hypergraph and Saxl graph
    are clearly identical for groups of base size 2. The succeeding text
    is devoted to the exploration of a selection of properties of
    Saxl hypergraphs, as outlined in the remainder of this section.

    \subsection{Complete Saxl hypergraphs}

    Burness and Giudici note that the transitive groups of base size 2
    with complete Saxl graphs are, by definition, the Frobenius groups.
    Although Freedman \etal observe that the situation for larger base sizes
    is complicated by the absence of a similar title and well-developed theory,
    the condition that $\saxh$ be complete (i.e.\ any $b(G)$ elements form a base)
    nonetheless proves sufficiently strong for us to prove a similar classification:

    \begin{theorem}
        \label{thm:complete}
        Let $G \le \Sym{(\Omega)}$ be a group of base size at least $2$.
        Its Saxl hypergraph is complete
        if and only if one of the following applies:
        \begin{enumerate}
            \item $G$ is Frobenius; \label{complete:frob}
            \item $G$ is an alternating or symmetric group
                in its natural action; \label{complete:alt}
            \item $G \in \{ \psl{2}{q}, \pgl{2}{q},
                \psl{2}{q^2}.\cyc{2} \not\le \psigmal{2}{q^2} \}$
                acts on the projective line; \label{complete:lin}
            \item $G$ is a Suzuki group $\suz{q}$ in its doubly transitive action;
                or \label{complete:lie_exc}
            \item $(G, |\Omega|) \in \{(\psl{2}{11}, 12), (\pgl{2}{11}, 12),
                (\MT{11}, 11), (\MT{12}, 12) \}$. \label{complete:spor}
        \end{enumerate}
    \end{theorem}

    The above result does not assume transitivity of $G$, and thus even
    represents a slight strengthening of the claim for Saxl graphs. We further
    note that completeness of $\saxh$ is a sufficient but unnecessary condition
    for $\saxg$ to be complete. The problem of determining which groups have
    this latter property, however, is not yet fully solved \cite{fhlk}.

    \subsection{\label{sec:goss_int}Common Neighbour Conjecture}

    As it turns out, the obvious generalisation of \cnj{cnc} to Saxl hypergraphs
    is equivalent to its generalisation to generalised Saxl graphs. There is
    in fact evidence for its equivalence to a still more powerful hypothesis.

    \begin{conjecture}
    \label{cnj:conj_1.4}
        Given a finite primitive group $G \le \Sym{(\Omega)}$ not
        containing ${\rm Alt} (\Omega)$, every pair of points
        $\alpha \neq \beta \in \Omega$ belong to some pair
        of edges $E_\alpha, E_\beta$ in $\saxh$ such that
        $|E_\alpha \cap E_\beta| = 1$.
    \end{conjecture}

    In \thm{cnc-edge-disj-6}, we prove this holds for primitive groups
    under a broad range of assumptions if \cnj{cnc} likewise holds
    for their generalised Saxl graphs. Crucially, for example, we show that \cnj{conj_1.4} holds under this assumption for all groups with base size at most 7. \\

    Some further results come from the exploration of the following
    apparently unnamed concept.

    \begin{definition}
        The {\rm \gosnum[n]} $g_n$ of a graph or hypergraph
        $\Gamma = (V, E)$ is the least number of common neighbours possessed
        by $n$ vertices of $\Gamma$; that is,
        \[ g_n = \min_{S \in V^k}{\left|
            \bigcap_{v \in S} N_{\Gamma} (v)
        \right|} \]
        if $1 \le n \le |V|$, and $0$ otherwise.
    \end{definition}

    As explained in \S\ref{sec:goss}, this provides another potential direction
    in which to extend \cnj{cnc}. The result below offers
    an initial development along these lines.

    \begin{theorem}
        \label{thm:goss}
        Let $G \le \Sym{(\Omega)}$ be a finite primitive permutation group. If
        $b(G) > 2$ and the CNC holds, then $g_2 \ge 2$. On the other hand,
        there exist groups of arbitrary base size $b(G) \ge 2$ such that
        $g_3 = 0$.
    \end{theorem}

    \subsection{Valency and Flag-Spanning Tours}

    A \textsl{flag-spanning tour} of a hypergraph is a walk which uses each flag
    $(v \in e, e)$ exactly once. As discussed in \S\ref{sec:valency}, such tours
    form a convenient partial analogue to Eulerian circuits in graphs,
    with similar valency-related criteria for their existence. A key focus of
    our study of the valency of $\saxh$ is therefore the question of which
    groups have a flag-spanning tour. We were able to achieve the following partial classification:

    \begin{theorem}
        \label{thm:flag-span}
            Let $G \le \Sym{(\Omega)}$ be a finite primitive permutation group
            such that $b(G) \in \{ 3, 4 \}$. Then one of the following holds:
            \begin{enumerate}
                    \item $\saxh$ has a flag-spanning tour;
                        \label{flag-span:yes}
                    \item $G$ is a primitive group with $| \Omega |$ odd,
                        as classified by Liebeck and Saxl
                        in \cite{odd_prim}, and $\saxh$ does not have
                        a flag-spanning tour; \label{flag-span:odd}
                    \item $G$ is a soluble group, $|\Omega|$
                            a power of two, and $|G|/|\Omega|$ is
                            not divisible by 4; \label{flag-span:solv}
                    \item $G$ is a group of product type, its base group
                        and associated point stabiliser appear
                        in \tbl{nofstbas}, and its top group is
                        a transitive (soluble) group of order
                        not divisible by 4; \label{flag-span:prod}
                    \item $G \in \{ \psl{2}{q}, \pgl{2}{q} \}$
                        acts on the projective line, where
                        $q \equiv 3 \textrm{ mod } 4$, and
                        $\saxh$ does not have a flag-spanning tour; or
                        \label{flag-span:lin1}
                    \item $G = \pgammal{2}{2^e}$ acts on pairs of
                        distinct projective points, $e$ is
                        a square-free odd integer, and $\saxh$ does not
                        have a flag-spanning tour. \label{flag-span:lin2}
            \end{enumerate}
    \end{theorem}
    \begin{table}
        \begin{minipage}{0.4\textwidth}
        \begin{tabular}{ll}
        \toprule
            $G$ & $H$ \\ 
            \midrule
            $\psl{3}{3}$ & $\cyc{13} {:} \cyc{3}$ \\ 
            $\psl{3}{3}.2$ & $\cyc{13} {:} \cyc{6}$ \\ 
            $\psl{3}{4}.\sym{3}$ & $\cyc{7} {:}
                (\cyc{3} \times \sym{3})$ \\ 
            $\psl{3}{4}.\dih{12}$ & $\sym{3} \times
                (\cyc{7} {:} \cyc{6})$ \\ 
            $\sym{7}$ & $\cyc{7} {:} \cyc{3}$ \\
            $\psl{2}{q}$ & $q {:} \left( \frac{q-1}{2} \right)$ \\
            ($q \equiv 3 \textrm{ mod } 4$) & \\
            \bottomrule
        \end{tabular}
        \end{minipage}
        \begin{minipage}{0.55\textwidth}
        \caption{\label{tbl:nofstbas}Potential base groups $G$
            with point stabilisers $H$ for groups $K$ of product type
            such that $b(K) \in \{ 3, 4 \}$ and $\saxh[K]$ has
            no flag-spanning tour.}
        \end{minipage}
    \end{table}

    Some additional comments on the contents of \thm{flag-span}
    appear at the end of \S\ref{sec:valency}.

    \subsection{Organisation}

    We begin \S\ref{sec:prelim} with some basic notation and
    hypergraph-related definitions, before collating various
    elementary observations about Saxl hypergraphs and
    further definitions in \S\ref{sec:lem212} and \S\ref{sec:arcs}.
    The first major topic, complete Saxl hypergraphs, then appears
    in \S\ref{sec:completeness}; the second, extensions of the CNC
    to hypergraphs, in \S\ref{sec:cnc}; and the last,
    valency-related results (including flag-spanning tours),
    in \S\ref{sec:valency}.

    \subsection*{Acknowledgements}
    
    
    This research was supported by an Australian Research Council Discovery
    Early Career Researcher Award (project number DE230100579).

    \section{\label{sec:prelim}Preliminaries}

    Let us first fix some basic notation and terminology. The indicator function
    $\ind{P}$ represents 1 when $P$ holds, and 0 otherwise. For a prime power $q$,
    the field of $q$ elements is denoted by $\mathbb{F}_q$, its group of units
    by $\mathbb{F}_q^{\times}$, and the vector space of dimension $d > 0$
    over $\mathbb{F}_q$ by $\vsp{d}{q}$. Our notation for many groups is standard, e.g. $\cyc{n}, \dih{n}$ for the cyclic and dihedral groups of
    order $n$, although it is noteworthy that the additive group of $\vsp{d}{q}$
    is itself written $\vsp{d}{q}$. An extension of a group $B$ by $A$ is
    denoted $A.B$, where $A$ is the normal subgroup, or sometimes $A {:} B$
    to highlight that the extension splits. Direct products are written
    $A \times B$. \\

    \label{onan+scott}
    Per the O'Nan--Scott Theorem \cite[Thm.\ 4.1A]{perm_grp},
    finite primitive groups fall into five classes. We term these
    the \textsl{affine},
    \textsl{almost simple}, \textsl{diagonal}, \textsl{product}, and
    \textsl{twisted wreath}-type groups. A brief description of all but the twisted wreath type groups, whose structure is more complex, is given in Table \ref{ons_types}, following \cite[Table 1]{prodsoltab}. In the particular case of
    primitive classical groups, the maximal subgroup type notation of
    Kleidman and Liebeck \cite{k+l} is employed to label point stabilisers.
    For groups
    $G \le L \wr P$ of product or diagonal type, meanwhile, the terms
    \textsl{base} and \textsl{top} group are used to refer to $L$ and $P$,
    respectively.

    \begin{table}[h]
    \renewcommand{\arraystretch}{1.2}
        \begin{tabular}{lp{9cm}}
        \toprule
            \textbf{Type} & \textbf{Description} \\ \midrule
            Affine & $G = V {:} H \le AGL(V)$,
                $H \le \ling{}{V}$ irreducible \\ 
            Almost simple & $T \le G \le \Aut{(T)}$ \\ 
            Diagonal & $T^k \le G \le T^k.({\rm Out}(T) \times P)$,
                $P \le \sym{k}$ primitive, or $k = 2$, $P = 1$ \\ 
            Product & $G \le L \wr P$, $L$ primitive of
                almost simple or diagonal type, $P \le \sym{k}$
                transitive \\ 
            Twisted wreath & \\
            \bottomrule
        \end{tabular}
        \caption{\label{ons_types} The types of finite primitive groups in the O'Nan-Scott theorem.}
    \end{table}

    \subsection{\label{sec:hyper}Hypergraphs}

    As noted in \S\ref{sec:intro}, we consider a hypergraph $G$ to be a pair
    composed of a vertex set $V$ and a (hyper)edges set
    $E \subseteq \mathcal{P} (V)$. Despite several inequivalent definitions of
    the term appearing in the literature \cite{hyper_rev}, the differences are
    immaterial insofar as our application to Saxl hypergraphs is concerned,
    which is henceforth almost their exclusive purpose. The Saxl hypergraph and
    generalised Saxl graph of a group $G$ are denoted by $\saxh$ and $\saxg$,
    respectively. \\

    Most of the terminology to be used in reference to hypergraphs
    simply generalises well-known graph-theoretic nomenclature
    in a straightforward fashion. For instance, two distinct vertices of
    a hypergraph $H = (V, E)$ are said to be \textsl{adjacent} if they belong
    to a common edge, while a \textsl{walk} of length $s$ is any sequence
    $v_0, E_1, v_1, \ldots , E_s, v_s$ of $v_i \in V$ and $E_i \in E$ such that
    each vertex lies in the neighbouring edge(s). We note only that
    the \textsl{degree} of a vertex refers to the number of edges containing it,
    a value often unequal to the size of its neighbourhood, and that
    a \textsl{partial hypergraph} is a hypergraph $(V' \subseteq V,
        E' \subseteq E)$ --- the analogue of a subgraph --- alongside
    some more novel concepts. \\

    A \textsl{$k$-uniform} hypergraph is one in which all edges have
    equal cardinality $k$. In particular, a $2$-uniform hypergraph is
    merely a simple graph. By analogy with this restricted case,
    a $k$-uniform hypergraph is called \textsl{complete} if it has
    all possible sets of $k$ vertices as edges. \\

    The \textsl{2-section} $[H]_2$ of a hypergraph $H = (V, E)$, meanwhile,
    is the graph on $V$ with edges between neighbours in $H$. It is clear
    that $[H]_2 = H$ in the event that $H$ is itself a graph. \\

    Further definitions will be introduced in the sections of the paper
    where they are used. For a more comprehensive introduction to hypergraphs,
    the reader is advised to consult Bretto \cite{bretto}, especially
    chapters 1, 2, and 7.

    \subsection{\label{sec:lem212}Basic Results}

    An extremely valuable observation is immediate upon considering
    the relevant definitions:

    \begin{lemma}
        \label{lem:2sec}
        For any Saxl hypergraph $\saxh = (V, E)$, there exists a walk of length $s$
        between $v, w \in V$ in $H$ if and only if there exists such a walk
         in the generalised Saxl graph $\saxg=[\saxh]_2$. 
        A similar duality extends to any property
        of graphs and hypergraphs defined in terms of paths (e.g.\ connectedness,
        neighbourhoods, diameters).
    \end{lemma}

    Some further useful results are given in the lemma below, including
    generalisations of the properties of Saxl graphs highlighted by Burness and Giudici\cite[Lemma~2.1]{burngiu}.
    Recall that a permutation group $G \le \Sym{(\Omega)}$ is termed \textsl{$k$-homogenous} if it
    acts transitively on the $k$-element subsets of $\Omega$.

    \begin{lemma}
        \label{lem:lem21}
        Let $G \le \Sym{(\Omega)}$ be a finite transitive permutation group.
        Suppose $b(G) \ge 2$.

        \begin{enumerate}
            \item $\saxh$ is $b(G)$-uniform. \label{lem21:uniform}
            \item  \label{lem21:aut} $G \le \Aut{(\saxh)}$ acts transitively on the set of
                vertices of $\saxh$, and consequently $\saxh$ is
                a vertex-transitive hypergraph with no isolated vertices.
            \item If $G$ is primitive, then $G$ is connected. \label{lem21:prim}
            \item The hypergraph $\saxh$ is edge-transitive --- indeed, complete
                --- if $G$ is $b(G)$-homogeneous.  \label{lem21:hom}
            \item If $K\leq G$ and $b(K) = b(G)$, then $\saxh$ is a partial hypergraph of $\saxh[K]$. \label{lem21:subg}
            \item $\alpha, \beta \in \Omega$ have the same set of neighbours in $\saxh$
                whenever $G_\alpha = G_\beta$. \label{lem21:eqstab}
        \end{enumerate}
    \end{lemma}
    \begin{proof}
        Apart from \ref{lem21:uniform}, which is obvious, the arguments are
        more or less analogous to the base 2 case. The transitivity of
        $G$ immediately yields \ref{lem21:aut}. For \ref{lem21:prim},
        it suffices to apply Lemma~2.5 of Freedman \etal\cite{fhlk}
        in conjunction with \lem{2sec}. Item \ref{lem21:hom} follows
        from the definition of $b(G)$-homogeneity, while
        \ref{lem21:subg} is a direct consequence of the observation that
        any base for $G$ is a base for $K \le G$ too. Finally, to show
        \ref{lem21:eqstab}, note that the bases for $G$ of size $b(G)$
        containing $\alpha$ (respectively, $\beta$) are obtained
        by adding $\alpha$ (resp.\ $\beta$) to the bases
        for $G_\alpha = G_\beta$ of size $b(G) - 1$.
    \end{proof}


    \subsection{\label{sec:arcs}Properties of arcs}

    Arcs in Saxl graphs are considered by Burness and Giudici
    \cite[Lemma~2.1]{burngiu}, an investigation extended
    to generalised Saxl graphs by Freedman \etal\cite{fhlk}. Unfortunately,
    the situation for Saxl hypergraphs is slightly complicated
    by the lack of clear and convenient analogue to the concept of an arc. \\

    The usual generalisation, defined by analogy
    with a directed graph, is a \textsl{hyperarc}
    in a directed hypergraph: a partition of an edge
    into two non-empty sets of ``sources'' and ``sinks''.
    Since such partition is not unique for an edge of
    more than two vertices, however, arc-transitivity and
    other related definitions become convoluted or even uninteresting.
    We therefore present two alternative generalisations: an \textsl{$s$-arc}
    ($s \ge 1$), as defined by Mansilla and Serra \cite{mans_serr}, is a walk
    of length $s$ in which any three consecutive vertices or
    two consecutive edges are distinct; we will define a \textsl{ray}, by contrast, as a single edge equipped
    with an ordering of its constituent vertices. Some easy initial observations
    may then be established.

    \begin{lemma}
        \label{lem:arcs}
        Let $G$ be a finite group.

        \begin{enumerate}
            \item If $\saxh$ is $s$-arc transitive, it is also $r$-arc transitive
                for all $1 \le r < s$. \label{arcs:arc_ext}
            \item 1-arc transitivity implies edge transitivity of $\saxh$
                and arc transitivity of $\saxg$. \label{arcs:1}
            \item Ray transitivity of $\saxh$ implies arc transitivity
                of $\saxg$. \label{arcs:ray}
            \item $G$ acts semiregularly on the rays of $\saxh$.
                \label{arcs:semireg}
        \end{enumerate}
    \end{lemma}

    \begin{proof}
        For \ref{arcs:arc_ext}, let
        $v_0, E_1, v_1, E_2, \ldots , E_r, v_r$
        and $v_0', E_1', v_1', E_2', \ldots , E_r', v_r'$ be
        any two $r$-arcs of $\saxh$ ($1 \le r < s$). Choosing
        suitable vertices and edges of $\saxh$ (as is always
        possible, since the irredundancy of minimal bases ensures
        any non-isolated vertex has at least 2 neighbours), we may extend
        these to $s$-arcs $v_0, E_1, \ldots , E_s, v_s$
        and $v_0', E_1', \ldots , E_s', v_s'$; by $s$-arc
        transitivity, there then exists an automorphism
        $\varphi$ of $\saxh$ that maps the former to the latter,
        and thus our arbitrary first $r$-arc to our arbitrary
        second $r$-arc. This gives $r$-arc transitivity.
        Parts \ref{arcs:1} and \ref{arcs:ray} are similar: here, any edge
        of $\saxh$ or pair of neighbouring vertices in
        $\saxg$ (which, by \lem{2sec}, are neighbours
        in $\saxh$ also) belong to a 1-arc, while
        any two neighbours in $\saxg$ (hence $\saxh$)
        extend to a ray. \\

        Finally, \ref{arcs:semireg} follows from the definition of
        a base and the fact that rays are ordered.
    \end{proof}

    \begin{remark}
            Note that $G$ need not act semiregularly
            on the $s$-arcs of $\saxh$. For instance, if
            some $g \neq 1 \in G$ fixes points $\alpha, \beta$
            in an edge $E$ while permuting the remaining
            points of $E$ --- as occurs for any edge when
            $G$ is a symmetric group --- then the $1$-arc
            $\alpha, E, \beta$ has a non-trivial stabiliser.
    \end{remark}

    Per parts (i-iii) of \lem{arcs}, the completion of the classification
    of arc-transitive generalised Saxl graphs begun in \cite{fhlk} will have
    implications for the possible $s$-arc- and ray-transitive
    Saxl hypergraphs. 

    \section{\label{sec:completeness}Complete Saxl hypergraphs}

    To fix our ideas, let $G \le \Sym{(\Omega)}$ be any permutation group
    (not necessarily transitive). We say $G$ is $\komp{n}$ if
    $b(G) = n$ and any $n$ points of $\Omega$ form a base for $G$.
    For $n = b(G) \ge 2$, this property is clearly equivalent to $G$ having
    a complete Saxl hypergraph. Some key results can
    then be stated as follows.

    \begin{lemma}
        \label{lem:complete}
        Let $G \le \Sym{(\Omega)}$ be as above, and $n \ge 0$ an integer.
        Then $G$ is $\komp{n}$ if and only if
        $(G, n) \neq (1, 1)$ and, for all $\alpha \in \Omega$,
        the group $G_\alpha$ acting on $\Omega \backslash \{ \alpha \}$
        is $\komp{n-1}$. In particular, a group is
        $\komp{0}$ if and only if it is trivial,
        $\komp{1}$ if and only if it acts semiregularly, and
        $\komp{2}$ if and only if it is Frobenius.
    \end{lemma}

    \begin{proof}
        The bases of minimum size for $G$ containing a given $\alpha \in \Omega$,
        should they exist, are precisely the bases of minimum size for $G_\alpha$
        with $\alpha$ added. It follows that, if $G$ is $\komp{n}$, such bases of
        $G_\alpha$ (for any $\alpha \in \Omega$) have $n - 1$ points, and that
        any $n - 1$ points in $\Omega \backslash \{ \alpha \}$ form such a base;
        the group $G_\alpha$ acting on $\Omega \backslash \{ \alpha \}$ is
        therefore $\komp{n-1}$. Conversely, if the group
        $G_\alpha$ acting on $\Omega \backslash \{ \alpha \}$ is
        $\komp{n-1}$ for every $\alpha \in \Omega$, let
        $B$ be any non-empty subset of $\Omega$. Taking
        arbitrary $\alpha \in B$ in the hypotheses, we see that $B$ is a base for $G$
        whenever $|B| = n$ and further that $b(G)=n$, so that either $G$ is $\komp{n}$
        or $\emptyset$ is a base for $G$. The latter condition
        implies that $G = 1$, so that $G_\alpha = 1$ is
        $\komp{0}$ and $(G, n) = (1, 0+1)$ is eliminated
        by assumption. \\

        As for the remainder of the lemma, it is clear from
        the definition that there is a unique $\komp{0}$ group:
        namely, the trivial group. A $\komp{1}$ group is then
        a non-trivial group in which every point has
        a trivial stabiliser --- that is, one acting semiregularly.
        Finally, a $\komp{2}$ group $G$ is one in which
        each point stabiliser acts semiregularly on the other points.
        We claim that this criterion implies transitivity, so that it
        characterises Frobenius groups as stated. Indeed,
        at most one orbit of such $G$ can have the property that it
        contains fixed points of more than half the elements of $G$.
        On the other hand, since no orbit $\mathcal{O}$ of $G$ can
        contain only one point --- lest $G = G_{(\mathcal{O})}$ be
        both $\komp{2}$ and $\komp{1}$ ---
        our criterion ensures the transitive action of $G$
        on any $\mathcal{O}$ must be faithful and Frobenius.
        It follows from the Structure Theorem of Frobenius groups
        \cite[\S3.4]{perm_grp} that all orbits have
        the aforementioned property.
    \end{proof}

    \begin{corollary}
        \label{corr:complete}
        Suppose that the finite group $G$ is $\komp{n}$,
        where $n \ge 2$. Then $G$ is $(n-1)$-transitive.
    \end{corollary}
    \begin{proof}
        The proof is by induction on $n$. The case $n = 2$ is immediate
        from \lem{complete}. Hence, suppose the hypothesis holds
        for fixed arbitrary $n \ge 2$. Let $G \le \Sym{(\Omega)}$ be
        $\komp{n+1}$. By \lem{complete}, any point stabiliser
        $G_\alpha$ of $G$ is $\komp{n}$ in its action
        on $\Omega \backslash \{ \alpha \}$, and hence
        $(n-1) \ge 1$-transitive by the induction hypothesis. Since
        $| \Omega | \ge | b(G) | \ge 2+1$, we can therefore
        take $G_\alpha, G_\beta \le G$ with overlapping orbits
        $\Omega \backslash \{ \alpha \}$ and
        $\Omega \backslash \{ \beta \}$ which cover $\Omega$,
        establishing that $G$ is transitive with $(n-1)$-transitive
        point stabilisers. The result follows by induction.
    \end{proof}

    The multiple transitivity implied by \corr{complete} is
    an extremely powerful restriction on the possibilities
    for $\komp{n}$ groups in the case $n \ge 3$.
    It is well-known that the only $k$-transitive groups
    for $k \ge 6$ are $\sym{k+1}$ and $\alt{k+2}$
    in their natural actions, while all $k \ge 4$ are covered
    upon the addition of the Mathieu groups $\MT{n}$
    ($n \in \{ 11, 12, 23, 24 \}$);
    a complete classification for $k \ge 2$ is provided
    in Dixon and Mortimer \cite[\S7.7]{perm_grp}. As such, it is
    possible to validate the classification of groups
    with complete Saxl hypergraphs given in \thm{complete}.

    \begin{proof}[Proof of \thm{complete}]
        By \lem{complete}, case~\ref{complete:frob} accounts for
        $b(G) = 2$. We may therefore assume $b(G) \ge 3$, so that
        $G$ is one of the $2$-transitive groups in \cite{perm_grp}.
        It will be convenient to proceed by considering
        the various families of such groups in turn. \\

        The alternating groups $\alt{k}$ and symmetric groups
        $\sym{k}$ are sharply $(k-2)$- and $(k-1)$-transitive,
        respectively. It follows that any set of $k-2$ or $k-1$ points
        forms an irredundant base, proving that $\saxh$ is
        indeed complete in case~\ref{complete:alt} above. \\

        As for 2-transitive affine groups $G$, observe that an element
        in $G_0 \cap \ling{d}{q} \neq 1$ which fixes one or more vectors
        in $\vsp{d}{q}$ must also fix any linear combinations
        thereof. If $q \neq 2$, \lem{complete} and its corollary
        thus ensure that $G_{0, v}$ is not
        $\komp{n}$ for any $v \in \vsp{d}{q} \backslash \{ 0 \}$ and
        $n > 0$, so that $G$ is not $\komp{n}$
        for any $n = b(G) > 2$ and the proof is complete. If $q = 2$,
        a similar analysis with $G_{0, v, w}$ for $v \neq w \neq 0
            \in \vsp{d}{q}$ shows that $G$ cannot be
        $\komp{b(G)}$ if $b(G) > 3$. This proof is complete as well,
        as all 2-transitive $G$ with socle $\vsp{d}{2}$ for which
        $b(G) = 3$ fall under one of the cases listed in the theorem:

        \begin{itemize}
            \item If $\lins{d}{2} \le G_0 \le \gammal{d}{2}
                = \lins{d}{2}$, then either $d > 2$ ---
                in which case $G_0$ contains non-trivial automorphisms
                that fix $3 - 1 = 2$ vectors in $\vsp{d}{2}$ ---
                or $G \in \{ \cyc{2}, \sym{4} \}$, covered
                by case~\ref{complete:alt} above. \\
            \item If $\symp{d}{2} \le G_0 \le \gammal{d}{2}$,
                then either $d > 2$ --- in which case $G_0$ again
                contains 
                a non-trivial automorphism that fixes
                $3 - 1 = 2$ vectors in $\vsp{d}{2}$ --- or
                $d = 2$ and $G \cong \sym{4}$. \\
            \item Otherwise, $(d, G_0)$ is one of $(4, \sym{6})$,
                $(4, \alt{6}{:}\cyc{2})$, $(4, \alt{7})$,
                $(6, \psigmau{3}{3} \cong \xlie{G}{2}{2})$, and
                $(6, \psu{3}{3})$. Computations confirm
                none of these correspond to $\komp{3}$
                groups.
        \end{itemize}

        Turning now to projective linear groups, observe that
        the stabiliser of two 1-dimensional subspaces leaves
        the plane they span invariant. In a space of $d\geq 3$ dimensions, this implies that such a stabiliser
        is not transitive on the remaining subspaces, and therefore
        by \corr{complete} that
        $\psl{d}{q}$ or any containing group is not
        $\komp{n}$ for any $n \ge 4$. On the other hand, three subspaces
        are insufficient to form a base
        for such a group, so that $b(G) \ge 4$. In the case $d = 2$,
        by contrast, the group $\pgl{2}{q}$ is 
        sharply 3-transitive and thus has complete 2- or 3-uniform
        $\saxh$. The same holds for $\psl{2}{q}$
        by part~\ref{lem21:subg} of \lem{lem21}, since this group
        contains a non-trivial element fixing a pair of complementary 1-dimensional subspaces whenever $q > 2$.
        It does not apply to either of these groups with the addition of
        standard field automorphisms, however, as the stabiliser of
        the 1-dimensional subspaces spanned by 2 vectors and their sum
        then fixes the linear combinations of those vectors
        with coefficients in the fixed field, but not others.
        The only other possibility
        in light of the fact that
        $|\pgl{2}{q} : \psl{2}{q}|  \le 2$
        is to extend $\psl{2}{q}$ with the composition of
        an involutory field automorphism and some
        $x \in \pgl{2}{q} \backslash \psl{2}{q}$. We hence
        obtain case~\ref{complete:lin}. \\

        For the 2-transitive actions of symplectic groups $\symp{2m}{2}$, the 2-point stabilisers have precisely two orbits
        on the remaining points, with lengths $2^{2m-2}$ and
        $2^{2m-2} \pm 2^{m-1}$ (see e.g. \cite[p.69]{kantor}). Neither action therefore gives rise
        to a complete Saxl hypergraph $\saxh$. \\

        To eliminate groups $G$ with socle $\psu{3}{q}$ acting
        on isotropic subspaces, first fix $q$ and a generator
        $\zeta$ of the multiplicative group $\mathbb{F}_{q^2}^{\times}$.
        The plane spanned by two linearly independent isotropic vectors
        $x, y$ then contains a third isotropic vector
        $x + \zeta^{\frac{q+1}{(2, q+1)}} \langle x, y \rangle y$.
        It follows by the same argument as for the projective linear
        groups that we need only consider the possibility that $G$
        is $\komp{3}$; since the stabiliser in $\psu{3}{q}$ of
        three coplanar 1-dimensional subspaces is non-trivial,
        this is not so. \\

        Turning to the Suzuki and Ree groups $\suz{2^{2n+1}}$ and
        $\ree{3^{2n+1}}$, it is noted
        in \cite{perm_grp} that there are permutation representations
        on the sets $\mathbb{F}_q^2 \cup \{ \infty \}$ (where $q = 2^{2n+1}$)
        and $\mathbb{F}_q^3 \cup \{ \infty \}$ (where $q = 3^{2n+1}$),
        respectively, with the pointwise stabilisers of
        $\{ 0, \infty \}$ consisting of the maps given by

        \[ n_\kappa : \, (\eta_1, \eta_2) \mapsto (
            \kappa \eta_1, \kappa^{\sqrt{2q} + 1} \eta_2
        ) \textrm{ or} \] \[ (\eta_1, \eta_2, \eta_3) \mapsto (
            \kappa \eta_1, \kappa^{\sqrt{3q} + 1} \eta_2,
            \kappa^{\sqrt{3q} + 2} \eta_3
        ) \]

        and $\infty \mapsto \infty$
        for each $\kappa \in \mathbb{F}_q^{\times}$. The corresponding
        stabilisers in their almost simple normalisers arise through
        the addition of standard field automorphisms $f \in \Aut{(\mathbb{F}_q)}$. It
        follows that any group with socle a Suzuki or Ree group has
        a base size of 3, for if $\zeta$ is
        a generator of $\mathbb{F}_q^{\times}$, the only map
        of the form $n_k f$ which satisfies

        \[ (1, \zeta) = n_\kappa f ((1, \zeta)) = (
            \kappa f(1), \kappa^{\sqrt{2q} + 1} f(\zeta)
        ) \textrm{ or} \] \[ (1, 1, \zeta) = n_\kappa f ((1, 1, \zeta)) = (
            \kappa f(1), \kappa^{\sqrt{3q} + 1} f(1),
            \kappa^{\sqrt{3q} + 2} f(\zeta)
        ) \]

        is the identity $n_1$. On the other hand, since

        \[ n_{-1} ((0, 1, 0)) = (-0, (-1)^{\sqrt{3q} + 1} \cdot 1, (-1)^{\sqrt{3q}+2} \cdot 0)
            = (0, 1, 0) \]

        for odd $q$, the 2-point stabiliser of a Ree group
        does not act semiregularly
        on the remaining points, so that these groups are not $\komp{3}$.
        The same holds for a Suzuki group extended
        with a field automorphism $f$, as $f ((1, 0)) = (f(1), f(0))$,
        but not for the Suzuki groups themselves: the existence
        of non-zero $(x, y) \in \mathbb{F}_q^2$ for which
        $n_\kappa ( (x, y) ) = (x, y)$ yields either $\kappa = x/x = 1$
        or $\kappa = (\kappa^{\sqrt{2q}+1})^{\sqrt{2q} - 1} =
        (y/y)^{\sqrt{2q} - 1} = 1$. This is precisely
        case~\ref{complete:lie_exc}. \\

        Finally, the sporadic 2-transitive groups can
        be examined computationally using the fact that,
        by \corr{complete}, a $\komp{n}$ group is
        sharply $n$-transitive or $(n-1)$-transitive
        with a semiregular $(n-1)$-point stabiliser. This
        produces the list in case~\ref{complete:spor}.
    \end{proof}

    \section{\label{sec:cnc}Connectedness properties}

    \lem{2sec} ensures that results on connectedness and diameter apply
    to $\saxg$ and $\saxh$ equally, as does a straightforward generalisation of
    the Common Neighbour Conjecture (\cnj{cnc}). We consider two potential avenues
    for further investigation of Saxl hypergraphs suggested by this conjecture.

    \subsection{\label{sec:edge-dis}Edge disjointness conditions}

    One option for strengthening the generalised Common Neighbour Conjecture
    (CNC) arises from a more explicit formulation thereof:

    \begin{conjecture}[Generalised CNC]
        \label{cnj:cnc-exp}
        If $G \le \Sym{(\Omega)}$ is a finite permutation group of
        base size $b(G) \ge 2$, then for any $\alpha \neq \beta \in \Omega$,
        there exist a third point $\gamma \neq \alpha, \beta
            \in \Omega$ and edges $E_\alpha, E_\beta$ of $\saxh$
        (or $\saxg$, in the original context) such that
        $\alpha, \gamma \in E_\alpha$ and $\beta, \gamma \in E_\beta$.
    \end{conjecture}

    In the case of a hypergraph, one may further sensibly ask whether
    $E_\alpha$ and $E_\beta$ can be chosen as distinct edges, or indeed
    to intersect in \textsl{only} $\gamma$. \\

    The addition of the weaker first requirement is easily demonstrated
    to not logically strengthen the conjecture. It will in fact be
    convenient to prove slightly more.

    \begin{theorem}
        \label{thm:cnc-edge-dist}
        Suppose that a finite group $G \le \Sym{(\Omega)}$, vertices
        $\alpha, \beta \in \Omega$, and edges $E_\alpha, E_\beta$
        of $\saxh$ satisfy the conditions of \cnj{cnc-exp}. Then
        there are edges $E_\alpha', E_\beta'$ of $\saxh$ satisfying
        the same conditions such that
        $\alpha, \beta \notin E_\alpha' \cap E_\beta'$ and
        $| E_\alpha' \cap E_\beta' | \le | E_\alpha \cap E_\beta |$.
        In particular, we may take $E_\alpha \neq E_\beta$
        in the conjecture.
    \end{theorem}
    \begin{proof}

        Note that either
        $\alpha \in E_\beta$ or $\alpha \notin E_\beta$.
        In the latter case, set $E_\beta' = E_\beta$.
        Otherwise, since $E_\beta$ is a base for $G$ of
        minimal possible order, there exists $g \in G$
        which fixes $E_\beta \backslash \{ \alpha \}$ pointwise,
        but not $\alpha$. The set $E_\beta' = E_\beta^g$ thus
        forms a base for $G$ of size $b(G)$ --- an edge of $\saxh$
        --- which contains $\beta$ and $\gamma \neq \alpha$, but not $\alpha$.
        Moreover, the fact that $\alpha \in E_\alpha$
        ensures $| E_\alpha \cap E_\beta' | = | (
            E_\alpha \cap (E_\beta \cup \{ \alpha^g \}
        )) \backslash \{ \alpha \} | \le | E_\alpha \cap E_\beta |$.
        An analogous argument with $E_\alpha, E_\beta'$ yields
        $E_\alpha'$; the result follows from the fact that
        $|E_\alpha' \cap E_\beta'| \le |E_\alpha \cap E_\beta'| \le
            |E_\alpha \cap E_\beta|$
        and $\alpha \notin E_\beta'$, $\beta \notin E_\alpha'$.

                The final claim follows from the fact that
        $\alpha \notin E_\alpha' \cap E_\beta'$, whereas
        $\alpha \in E_\alpha'$. 
    \end{proof}

    As for the stronger second requirement, the straightforward
    implication that $2b(G) - 1 = |E_\alpha| + |E_\beta| -
        |E_\alpha \cap E_\beta| = |E_\alpha \cup E_\beta| \le
        |\Omega|$ yields an easy bound
    $b(G) \le \frac{\left| \Omega \right| + 1}{2}$, from which it follows
    that the symmetric and alternating groups of degree at least 4 and 6
    (respectively) are counterexamples in their natural actions.
    On the other hand,  a result of Halasi
    \etal\cite[Corollary~1.3]{pyber_exp} states that any other
    primitive group of degree $n$ has a base of size
    at most $\max{\{ \sqrt{n}, 25 \}}$; given that $\sqrt{n} \le (n+1)/2$ and
    that computational experiments with the GAP Primitive Groups Library
    \cite{primgrp} show there are no further counterexamples of degree
    less than $25 \cdot 2 - 1 = 49$, we conjecture that these are
    the only counterexamples. The following is a more explicit version of \cnj{conj_1.4}.

    \begin{conjecture}
        \label{cnj:cnc-edge-disj}
        For any finite primitive group $G \le \Sym{(\Omega)}$ of
        base size at least 2, exactly one of the following holds:
        \begin{itemize}
            \item $G$ is $\sym{n}$ in its natural action, where
                $n \ge 4$;
            \item $G$ is $\alt{n}$ in its natural action, where
                $n \ge 6$; or
            \item Given any $\alpha \neq \beta \in \Omega$, there
                exist edges $E_\alpha, E_\beta$ of
                $G$'s Saxl hypergraph $\saxh$ such that
                $\alpha \in E_\alpha$, $\beta \in E_\beta$,
                and $|E_\alpha \cap E_\beta| = 1$.
        \end{itemize}
    \end{conjecture}
    \begin{remark}
        \label{rmk:altd}
        The first two cases of 
         \cnj{cnc-edge-disj}
        cover all symmetric and alternating groups in their natural actions except 
        $\sym{3}$,
        $\alt{4}$, and $\alt{5}$, where sharp
        2-, 2-, and
        3-transitivity (respectively) make it straightforward
        to verify the conjecture.
    \end{remark}

    We now present some strong evidence for \cnj{cnc-edge-disj}. Recall that a primitive action of
    an almost simple group $G$ is called \textsl{standard}
    if $G$ has socle $\alt{k}$ and permutes subsets or partitions
    of $\{ 1, \ldots , k \}$, or $G$ is a classical group acting
    on subspaces or complementary pairs of subspaces.

    \begin{theorem}
        \label{thm:cnc-edge-disj-6}
        Let $G$ and $\Omega$ satisfy \cnj{cnc-exp}
        (the generalised CNC). If
        any of the following apply, they further satisfy
        \cnj{cnc-edge-disj}:
        \begin{enumerate}
            \item $b(G) \le 7$; \label{ced6:base_size}
            \item $G$ is of affine type; \label{ced6:aff}
            \item $G$ is an almost simple group in
                a non-standard action; or \label{ced6:simp}
            \item $G$ is a group of diagonal type with socle $T^k$ and
                top group $P \notin \{ \alt{k}, \sym{k} \}$,
                or $k < |T|^6 - 1$. \label{ced6:diag}
            \item $G$ is of product type, with its base group
                almost simple in a non-standard action and its
                top group soluble or primitive (but not
                a symmetric or alternating group in its natural action);
                \label{ced6:prod-simp}
            \item $G$ is of product type, possesses a soluble or primitive
                top group which is not a symmetric or alternating group
                in its natural action, has a base group $L$
                of diagonal type with socle $T^k$, and either
                $k < |T|^5 - 1$ or the base group of $L$
                is not one of $\alt{k}, \sym{k}$; or
                \label{ced6:prod-diag}
            \item $| \Omega | \le 128$. \label{ced6:omega}
        \end{enumerate}
    \end{theorem}
  \begin{table}[h!]
\caption{\label{tbl:detritus}Primitive groups $G$ of degree
            at most than 128 for which we do not verify \cnj{cnc-edge-disj}.}
        \renewcommand{\arraystretch}{1.2}
        \begin{tabular}{llp{3.28cm}l}
        \toprule
            $G$ & $| \Omega |$ & $b(G)$ & Comments \\ \midrule
            $\symp{8}{2}$ & $2^7 - 2^3$ & 8 & 2-transitive action on cosets
                of ${\rm GO}_8^- (2)$ \\ 
            $G \le \sym{m} \wr \sym{2}$ & $m^2$ & $m$ ($G = \sym{m} \wr \sym{2}$),
                $m - 1$ (otherwise) & $8 \le m \le 11$; $G$
                contains $\alt{m} \wr \sym{2}$ \\
            $\alt{m}, \sym{m}$ & $\binom{m}{2}$ & $\lceil \frac{2(m-1)}{3} \rceil$
                ($G = \sym{m}$), $\lceil \frac{2(m - 2)}{3} \rceil$ ($G = \alt{m}$) &
                $12 \le m \le 16$\\
            \bottomrule
        \end{tabular}
    \end{table}

    The proof will be presented in stages, the first of which
    relies on the following lemma.
    \begin{lemma}
\label{lem:ced6-pre}
Suppose $G \le \Sym(\Omega)$ is a permutation group with subgroups $H, K \le G$.
Let $\delta \in \Omega$, and let the sets
\[
M_K \subseteq F_H \subseteq \Omega \setminus \{\delta\}, \quad
M_H \subseteq F_K \subseteq \Omega \setminus \{\delta\}
\]
be non-empty and satisfy the following conditions:
\begin{enumerate}
    \item $H$ fixes every point in $F_H$;
    \item $K$ fixes every point in $F_K$;
    \item $H$ leaves both $M_H \cup \{\delta\}$ and $F_K \cup \{\delta\}$ invariant, and acts transitively on $M_H \cup \{\delta\}$;
    \item $K$ leaves both $M_K \cup \{\delta\}$ and $F_H \cup \{\delta\}$ invariant, and acts transitively on $M_K \cup \{\delta\}$.
\end{enumerate}
Then there exists a subgroup $L \le G$ that acts as the alternating group on $M_H \cup M_K \cup \{\delta\}$ and fixes every point in $F_H \cup F_K \cup \{\delta\}$ otherwise.
\end{lemma}
    \begin{proof}
        For any points $\alpha \in M_H$ and $\beta \in M_K$,
        the transitivity conditions on $H$ and $K$ yield
        $h \in H$ and $k \in K$ such that $\delta^h = \alpha$,
        $\delta^k = \beta$. In particular, neither $h$ nor $k$
        fixes $\delta$.
        It follows
        from the properties of $H$ and $K$ that the action of
        the commutator $[h, k] \in G$ on $\Omega$ is such that

        \begin{align*}
            \theta^{[h, k]} = \theta^{h^{-1} k^{-1} hk} =
            (\theta^{k^{-1}})^{hk} = \theta^{k^{-1} k} = \theta
            & \qquad \textrm{if } \theta \in F_H \backslash
            \{ \delta^k = \beta \}, \\
            \theta^{[h, k]} =
            (\theta^{h^{-1}})^{k^{-1} hk} = \theta^{h^{-1} hk} =
            \theta^k = \theta
            & \qquad \textrm{if } \theta \in F_K \backslash
            \{ \delta^h = \alpha \}, \\
            \delta^{[h, k]} =
            (\delta^{h^{-1}})^{k^{-1} hk} = \delta^{h^{-1} hk} =
            \delta^k = \beta, & \\
            \beta^{[h, k]} =
            (\delta^k)^{h^{-1} k^{-1} hk} = \delta^{k k^{-1} hk} =
            (\delta^h)^k = \delta^h = \alpha, & \qquad \textrm{and} \\
            \alpha^{[h, k]} =
            \delta^{hh^{-1} k^{-1} hk} = (\delta^{k^{-1}})^{hk} =
            \delta^{k^{-1} k} = \delta. &
        \end{align*}

        It thus induces the 3-cycle $( \delta \, \beta \, \alpha )$
        on $F_H \cup F_K \cup \{ \delta \}$. Taking $L$ to be
        the subgroup of $G$ generated by these elements
        as $\alpha$ and $\beta$ range over $M_H$ and $M_K$,
        respectively, the induced permutations act
        as the alternating group on $M_H \cup M_K \cup \{ \delta \}
            \subseteq F_H \cup F_K \cup \{ \delta \}$
        for the desired result.
    \end{proof}

    \begin{lemma}
        \label{lem:ced6-6}
        \thm{cnc-edge-disj-6} holds under
        hypothesis~\ref{ced6:base_size} (i.e.\ when $b(G) \le 7$).
    \end{lemma}

    \begin{proof}
        Suppose that $G$ is a counterexample. By assumption, there exist
        $\alpha \neq \beta \in \Omega$ for which no edges
        $E_\alpha \owns \alpha, E_\beta \owns \beta$ of
        $\saxh$ satisfy
        $|E_\alpha \cap E_\beta| = 1$, but
        $E_\alpha \cap E_\beta \neq \emptyset$ for some edges;
        choose $E_\alpha, E_\beta$ so that
        $|E_\alpha \cap E_\beta| \ge 2$ is minimal.
        By \thm{cnc-edge-dist}, we may simultaneously require
        $\alpha, \beta \notin E_\alpha \cap E_\beta$. \\

        Now, fix $\gamma \in E_\alpha \cap E_\beta$. By construction,
        any edge for $G$ of the form $E_\beta \backslash \{ \gamma \}
            \cup \{ \delta \}$ contains $\beta$ and intersects
        $E_\alpha$ in $| (E_\alpha \cap E_\beta) \backslash \{ \gamma \}
            \cup (E_\alpha \cap \{ \delta \}) | \geq 
            |E_\alpha \cap E_\beta| - 1  \ge 1$ points;
        since $|E_\alpha \cap E_\beta| \neq 0$ was chosen to be minimal,
        this implies $\delta \in E_\alpha$. It follows that
        all $\delta$ for which $E_\beta \backslash \{ \gamma \}
            \cup \{ \delta \}$ is a base of $G$, or equivalently
        on which $G_{(E_\beta \backslash \{ \gamma \})}$ has
        a regular orbit, lie in $E_\alpha$. Interchanging $\alpha$ and
        $\beta$ yields an analogous result. \\

        The subsequent argument is split into three cases.
        First, suppose that one of
        $| G_{(E_\alpha \backslash \{ \gamma \})} |$ and
        $| G_{(E_\beta \backslash \{ \gamma \})} |$ is prime ---
        without loss of generality, let it be the latter. It follows that
        $H = G_{(E_\beta \backslash \{ \gamma \})}$ fixes every element
        on which it does not have a regular orbit. Hence,
        the previous paragraph establishes that the conditions of
        \lem{ced6-pre} are satisfied upon setting
        $F_H = \Omega \backslash E_\alpha$,
        $M_H = \gamma^H \backslash \{ \gamma \}$,
        $K = G_{(E_\alpha \backslash \{ \gamma \})}$,
        $F_K = E_\alpha \backslash \{ \gamma \}$, and
        $M_K = \gamma^K \backslash \{ \gamma \}$, so that there exists
        $L \le G$ which induces all even permutations of
        $M = M_H \cup M_K \cup \{ \gamma \}$ while fixing
        the remaining points of $F_H \cup F_K \cup \{ \gamma \} = \Omega$.
        Any maximal $\Delta \subseteq \Omega$ which can replace
        $M$ in this statement therefore contains at least 3 elements.
        But then $\Delta \cup \Delta^g \supseteq \Delta$ can also
        replace $M$ for any $g \in G$ such that $\Delta \cap \Delta^g
            \neq \emptyset$, implying $\Delta$ is a block of
        at least 3 points for the primitive group $G$ and thus that
        $G \le \Sym{(\Omega})$ contains all even permutations of
        $\Delta = \Omega$. In light of \rmk{altd}, this contradicts
        the choice of $G$ as a counterexample to \cnj{cnc-edge-disj}. \\

        Consequently, neither
        $| G_{(E_\alpha \backslash \{ \gamma \})} |$ nor
        $| G_{(E_\beta \backslash \{ \gamma \})} |$ is prime.
        Since $E_\alpha$ and $E_\beta$ are bases for $G$ of
        minimal size, these values must in fact be composite.
        In particular, they must be at least 4.
        Write $S_\alpha$ for the set of points
        in $E_\beta$ not fixed by, or in regular orbits of,
        $G_{(E_\alpha \backslash \{ \gamma \})}$, with
        $S_\beta$ defined analogously. \\

        We next assume
        $| S_\alpha | \le 1$ or $| S_\beta | \le 1$ --- as before,
        without loss of generality, let it be the latter.
        If $S_\beta$ contains a point, that point does not belong
        to a regular orbit under $G_{(E_\beta \backslash \{ \gamma \})}$,
        so that in any event
        $H = G_{(E_\beta \backslash \{ \gamma \} \cup S_\beta)}$ is
        non-trivial. It again follows in conjunction
        with the previous discussion that \lem{ced6-pre} applies
        upon setting
        $F_H = E_\beta \backslash (S_\alpha \cup \{ \gamma \}) \cup S_\beta$,
        $M_H = \gamma^H \backslash \{ \gamma \}$,
        $K = G_{(E_\alpha \backslash \{ \gamma \})}$,
        $F_K = E_\alpha \backslash (S_\beta \cup \{ \gamma \})$,
        and $M_K = \gamma^K \backslash \{ \gamma \}$, yielding
        $L \le G$ which induces all even permutations of
        $M_H \cup M_K \cup \{ \gamma \} \supset \gamma^K$
        while fixing the remaining points of
        $F_H \cup F_K \cup \{ \gamma \} \supset E_\alpha$.
        But then there is $g \in G$ which permutes the $4 - 1 = 3$ or
        more points of $\gamma^K \backslash \{ \gamma \}$ despite fixing
        the base $E_\alpha$ of $G$, for yet another contradiction. \\

        Finally, suppose that $b(G) \le 7$. According to the results
        thus far, we may partition $E_\alpha$ (and similarly $E_\beta$)
        into 3 sets: at least
        $| \gamma^{G_{(E_\beta \backslash \{ \gamma \})}} | \ge 4$ points
        on which $G_{(E_\beta \backslash \{ \gamma \})}$ has a regular orbit,
        at least $| (E_\alpha \cap E_\beta) \backslash \{ \gamma \} | \ge 1$
        points it fixes, and $| S_\beta | \ge 2$ other points. This accounts
        for at least $4+1+2 = 7$ points, so that equality must hold
        in all the preceding statements. Hence, write
        $S_\beta = \{ \delta, \delta' \}$. If the setwise stabiliser of
        $S_\beta$ in $G_{(E_\beta \backslash \{ \gamma \})}$ is
        non-trivial, then the argument in the previous paragraph
        again yields a contradiction; otherwise, 
        $H = G_{(E_\beta \backslash \{ \gamma \} \cup \{ \delta \})}$
        has a regular orbit on $\delta'$. Moreover, as
        $G_{(E_\beta \backslash \{ \gamma \})}$ does not fix or have
        a regular orbit on $\delta \in S_\beta$, the group $H$ is
        a non-trivial proper subgroup thereof. Hence, $|H|$
        must be the unique proper divisor 2 of 4. Write
        $(\delta')^H = \{ \delta', \delta'' \}$, noting
        $\delta'' \notin E_\alpha \cup E_\beta$. It follows
        $G_{(E_\alpha \backslash \{ \gamma \})}$ does not have
        a regular orbit on $\delta''$, and thus again that
        $K = G_{(E_\alpha \backslash \{ \gamma \} \cup \{ \delta'' \})}$
        is non-trivial. \\

        Bearing the above in mind, along with the fact that
        $H$ permutes $E_\alpha \cup \{ \delta'' \} =
        \gamma^{G_{(E_\beta \backslash \{ \gamma \})}} \cup
        \{ \delta, \delta', \delta'' \} \cup (E_\alpha \cap E_\beta)$,
        we may apply \lem{ced6-pre} with
        $F_H = \gamma^{G_{(E_\alpha \backslash \{ \gamma \})}} \backslash
            \{ \gamma \}$,
        $F_K = E_\alpha \backslash \{ \gamma \} \cup \{ \delta'' \}$,
        $M_H = \gamma^H \backslash \{ \gamma \}$, and
        $M_K = \gamma^K \backslash \{ \gamma \}$ to obtain
        $L \le G$ which induces all even permutations of
        $M_H \cup M_K \cup \{ \gamma \} = \gamma^K \cup \gamma^H$
        while fixing the remaining points of
        $F_H \cup F_K \cup \{ \gamma \} \supset E_\alpha \cup
            \gamma^{G_{(E_\alpha \backslash \{ \gamma \})}}$.
        In particular, writing $\delta_H$ and $\delta_K$
        for arbitrary points in $M_H$ and $M_K$, respectively,
        there is $l \in L$ which induces the permutation
        $( \gamma \; \delta_H \; \delta_K )$
        on $E_\alpha \cup \gamma^{G_{(E_\alpha \backslash \{ \gamma \})}}$.
        Conjugates of $l$ by elements in $G_{(E_\beta \backslash \{ \gamma \})}$
        then induce 3-cycles on this set containing
        $\delta_H \in E_\beta \backslash \{ \gamma \}$ and (collectively)
        all points in $\gamma^{G_{(E_\beta \backslash \{ \gamma \})}}$.
        The group generated by these conjugates therefore induces all even
        permutations of the five points
        $\gamma^{G_{(E_\beta \backslash \{ \gamma \})}}
            \cup \{ \delta_H \}$ while fixing all other points
        in $E_\alpha$, once more producing the contradiction that
        there is a non-trivial element in $G$ fixing a base.
        It follows as claimed that $G$ satisfies \cnj{cnc-edge-disj}.
    \end{proof}

    \begin{lemma}
        \label{lem:ced6-aff}
        \thm{cnc-edge-disj-6} holds under
        hypothesis~\ref{ced6:aff} (i.e.\ when $G$ is of affine type).
    \end{lemma}
    \begin{proof}
        Let $G \le \Sym{(\vsp{d}{q})}$ be a primitive group
        of affine type for which \cnj{cnc-exp} holds.
        Given any $v \neq w \in \vsp{d}{q}$,
        there consequently exist edges $E_v, E_w$
        of the Saxl hypergraph $\saxh$ such that
        $E_v \cap E_w \neq \emptyset$.
        By \thm{cnc-edge-dist}, it may be assumed
        without loss of generality that
        $v, w \notin E_v \cap E_w$. \\

        Now, the translation $u \mapsto u - w$
        produces a base $B = \{ x - w \mid x \neq w \in E_w \}$
        for $G_0$ of size $b(G) - 1$. Choosing arbitrary $u \in E_v \cap E_w$,
        let $B_2$ be a maximal linearly independent subset of
        $B' = \{ x \in \vspan{(B)} \mid x = u - w \textrm{ or } x + w \notin E_v \}$
        containing $u - w$. It is clear that $|B_2| \le |B| = b(G) - 1$;
        moreover, any $g \in G_0 \le \gammal{d}{q}$ which
        fixes $B_2$ pointwise must also fix $\vspan{(B_2)} = \vspan{(B')}$,
        so that if $\vspan{(B')} = \vspan{(B)}$ the set
        $E_w' = \{ x + w \mid x = 0 \textrm{ or } x \in B_2 \}$ is a base of
        minimal size for $G$. Since we have $w \in E_w'$ and $E_v \cap E_w' = \{ u \}$ by construction,
        \cnj{cnc-edge-disj} holds and we are done. \\
        
        It only remains to consider the possibility that
        $\vspan{(B')} \neq \vspan{(B)}$. Note that
        the former is a subspace of the latter vector space
        over $\mathbb{F}_q$, while the fact that
        $B$ is an irredundant base for $G_0 \le \gammal{d}{q}$
        guarantees its elements are linearly independent. Hence
        \[ q^{b(G) - 2} = q^{|B| - 1} \ge | \vspan{(B')} |
            \ge |B'| \ge | \vspan{(B)} | - |E_v | + 1 \]
        \[    = q^{b(G) - 1} - b(G) + 1. \]
        But this implies $b(G) - 1 \ge (q-1) q^{b(G) - 2} \ge 2^{b(G) - 2}$,
        from which it follows that $b(G) \le 3$. An application of
        \lem{ced6-6} completes the proof. \\

        It only remains to consider the possibility that
        $\vspan{(B')} \neq \vspan{(B)}$. Note that
        the former is a subspace of the latter vector space
        over $\mathbb{F}_q$, while the fact that
        $B$ is an irredundant base for $G_0 \le \gammal{d}{q}$
        guarantees its elements are linearly independent. Hence
        \[ q^{b(G) - 2} = q^{|B| - 1} \ge | \vspan{(B')} |
            \ge |B'| = | \vspan{(B)} | - |E_v - w| + 1 \]
        \[    = q^{b(G) - 1} - b(G) + 1. \]
        But this implies $b(G) - 1 \ge (q-1) q^{b(G) - 2} \ge 2^{b(G) - 2}$,
        from which it follows that $b(G) \le 3$. An application of
        \lem{ced6-6} completes the proof.
    \end{proof}

    \vspace{1.5pt}
    We now proceed with the remainder of the proof. 
    \begin{proof}[Proof of \thm{cnc-edge-disj-6}]
        The result largely follows from Lemmas~\ref{lem:ced6-6}
        and~\ref{lem:ced6-aff}. Sufficiency of
        hypothesis~\ref{ced6:simp} follows from the theorem of
        Burness \etal\cite[Corollary~1]{basesim} that $b(G) \le 7$
        for almost simple $G$ in a non-standard action. Likewise,
        in the case where $G$ is of diagonal type, Huang
        \cite[Theorem~3]{diagbase} establishes that the conditions of
        hypothesis~\ref{ced6:diag} ensure $b(G) \le 7$. Next
        suppose $G \le L \wr P$ is of product type. A result of Bailey
        and Cameron \cite{prodbase} states that $L \wr P$ possesses
        a base of size $t$ if and only if $L$ has at least $D(P)$
        regular orbits on $\Omega^t$, where the {\em distinguishing number}
        $D(P)$ is the least number of parts in a partition preserved
        by no non-trivial element of $P$; Seress \cite{dist_no1, dist_no2}
        proves that $D(P) \le 5$ for any top group $P$ satisfying the conditions
        in hypotheses~\ref{ced6:prod-simp} and~\ref{ced6:prod-diag},
        so that it is enough show the top group $L$
        of $G$ induces at least 5 regular orbits on $\Omega^7$
        under these conditions. Considering possible positions of repeated elements
        in $s \in \Omega^7$ shows as much holds if $b(H) \le 6$, which
        \cite{diagbase,prodbase} reveal is the case
        under the given conditions unless $L$ is $\MT{24}$
        in its natural action. In this exceptional case, using the fact that
        $\MT{24}$ is 5-transitive, one can nonetheless confirm
        computationally that there are 7 regular orbits
        on $\Omega^7$. \\

        It remains to handle case~\ref{ced6:omega}. Per Moscatiello and
        Roney-Dougal \cite[Theorems~1 and~5]{smlbasegrp}, a primitive group
        $G$ of degree $n$ satisfies $b(G) < \log_2{n} + 1$ unless
        $\alt{m}^r \le G \le \sym{m} \wr \sym{r}$ acts on $r$-tuples of
        $k$-element subsets of $\{ 1, \ldots , m \}$, the group
        $G = \MT{n}$ for $n \in \{ 12, 23, 24 \}$ has its natural action,
        $G = \symp{2m}{2}$ for $m \ge 4$ has its 2-transitive action
        on $n = 2^{2m - 1} - 2^{m - 1}$ points, or $G = {\rm AGL}_d (2)$.
        The second and fourth cases have $G$ almost simple and affine,
        respectively, and are thus covered by the above. In the third case,
        $n \le 128$ implies $m = 4$, which is listed in \tbl{detritus}.
        Finally, suppose the first case applies. If $m = 2$, then
        $\sym{m}^r \cong \vsp{r}{2}$
        and $G$ is affine. Otherwise, if $r > 2$, the condition that
        $\binom{m}{k}^r = | \Omega | \le 128$ demands
        $\binom{m}{k} \le \sqrt[3]{128} < 6$ and $r \le \log_3{128} < 5$.
        The top group of $G$ is thus soluble and its base group cannot have
        base size greater than 5, so that $b(G) \le 7$ by the argument
        for groups of product type in the previous paragraph. As for
        $r \in \{ 1, 2 \}$, we first mention a result in \cite{Caceres}
        (improving on a result in \cite{Halasi})
        that $b( \sym{m} ) = \lceil (2m - 2)/(k + 1) \rceil$ for $\sym{m}$ acting
        on sets of $k$ points whenever $2m \ge k^2 + k$. This yields
        $b( \sym{m} ) \le 6$ --- and thus, as before,
        $b(G) \le b( \sym{m} \wr \sym{r} ) \le 7$ --- for any $k < m/2$
        such that $\binom{m}{k}^2 \le 128$, with the exception of
        those listed in \tbl{detritus}. It similarly shows
        that $b( \sym{m} ) \le 7$ for any $k, m$ such that
        $\binom{m}{k} \le 128$, excluding cases in \tbl{detritus},
        those where $k = 1$, and $(k, m) = (4, 9)$; the second class
        are accounted for by \rmk{altd}, while an exact expression
        for the relevant base size obtained in \cite[Theorem~1.1]{altsubbase}
        establishes that $b(\sym{9}) \le 4$ in the latter case\footnote{
            :~Take $l = 4$ and $k = 2$ in del~Valle and Roney-Dougal's
                expression.
        }. \\
        
        To complete the proof, the cases listed in \tbl{detritus} are handled
        computationally. For convenience, we note that the base sizes of
        each group (which appear in \tbl{detritus}) follow from \cite{smlbasegrp},
        \cite{prodbase} and the sharp transitivity of alternting and symmetric groups
        in their natural actions, or \cite{Caceres}; moreover, up to conjugacy,
        it suffices to confirm \cnj{cnc-edge-disj} for a fixed $\alpha$ and
        one representative $\beta$ of each suborbit.
    \end{proof}

    \subsection{\label{sec:goss}\GosNum{}s}
    A second approach to generalisation of the CNC involves
    strengthening the constants in the statement
    ``any two vertices have a (i.e.\ 1) common neighbour''
    through the use of \gosnum{}s for (hyper)graphs. Recall from \S\ref{sec:goss_int} that the $n$th \gosnum{} $g_n$ is the least number of common neighbours possessed by $n$ vertices.
    As the relevant concept does not appear to have been
    named in the previous literature, we begin
    with some elementary properties. \\

    \begin{lemma}
        \label{lem:gosnum}
        Let $\Gamma$ be a finite graph or hypergraph
        and $i \le j$ positive integers. Then the \gosnum{}s of $\Gamma$
        satisfy $g_i \ge g_j$. Furthermore,
        $g_i + i \ge g_j + j$ if $g_j > 0$.
    \end{lemma}

    \begin{proof}
        That $g_i \ge g_j$ is trivial. As for the second claim,
        it suffices to show that $g_i + i \ge g_{i+1} + i + 1$
        whenever $g_{i+1} > 0$.  This follows from the fact
        that any $i$ vertices of $\Gamma$ and one of their
        $g_i \ge g_{i+1} \ge 1$ common neighbours form a set of size
        $i + 1$.
    \end{proof}

    By \lem{2sec}, the \gosnum{}s of a group's generalised Saxl graph
    and Saxl hypergraph are equal; these will simply be
    referred to as the \gosnum{}s of the group for brevity. The CNC can
    now be rephrased as the assertion that $g_2 > 0$
    for any finite primitive group $G$. Except in Burness and
    Giudici's original case $b(G) = 2$, interestingly, one can indeed say more:

    \begin{theorem}
        \label{thm:g2}
        Let $G \le \Sym{(\Omega)}$ be a finite group of base size
        $b(G) \ge 3$. Then $g_2 \neq 1$. In particular, if $G$ is
        primitive and the CNC holds, $g_2 \ge 2$.
    \end{theorem}
    \begin{proof}
        Suppose to the contrary. A counterexample yields
        $\alpha, \beta \in \Omega$ with a unique common neighbour
        $\gamma$ in $\saxh$. Fixing edges $\{ \alpha, \gamma, \alpha_2,
            \ldots , \alpha_{b(G)-1} \}$ and $\{ \beta, \gamma,
            \beta_2, \ldots , \beta_{b(G)-1} \}$ of $\saxh$,
        there are two cases to consider. \\

        If $G_{\alpha, \beta, \gamma} = G_{\alpha, \beta}$,
        it follows that $G_{\alpha, \beta} = G_{\alpha, \beta, \gamma}
            \le G_{\alpha, \gamma}$ and thus that
        $\{ \alpha, \beta, \alpha_2, \ldots , \alpha_{b(G)-1} \}$
        is a base for $G$. Both $\alpha$ and $\beta$ therefore have
        $\alpha_2 \neq \gamma$ as a neighbour: a contradiction.
        On the other hand, if $G_{\alpha, \beta, \gamma} < G_{\alpha, \beta}$,
        the multiple elements of $\gamma^{G_{\alpha, \beta}}$
        are all neighbours for $\alpha$ and $\beta$, another contradiction.
        The result follows.
    \end{proof}

    \begin{remark}
        The hypothesis $b(G) \ge 3$ is necessary to exclude
        the trivial counterexample $\sym{3}$. This extends
        to an infinite family of transitive counterexamples
        $\sym{3}^k$. Further 2-transitive counterexamples
        are impossible, for such groups have $\saxg$ a complete graph; the question of the existence of primitive counterexamples is open.
    \end{remark}

    Turning to $g_n$ for larger $n$, experiments with
    the GAP Primitive Groups library \cite{primgrp} suggest that
    $g_n > 0$ usually holds for values of $n$ much larger than 2.
    It is tempting to conjecture that this is so in general
    when $b(G) > 2$, or even that $g_{b(G)} > 0$ in all but
    perhaps a small number of cases; however, the following
    result confirms that the choice of 2 vertices
    in the generalised CNC is optimal.

    \begin{theorem}
        \label{thm:gos-aff}
        Let $q$ be an odd prime power and $n = q^k - 1$, where
        $k \ge 1$ is an integer. Then the primitive affine group
        $G = \vsp{n-1}{q}{:}\sym{n}$ satisfies $b(G) = k+1$ and
        $\gamma_3 = 0$.
    \end{theorem}
    \begin{proof}
        It will be convenient to treat $N = \vsp{n - 1}{q}$
        as the deleted permutation module for $\sym{n}$
        contained in $\vsp{n}{q}$. Primitivity of $G$ follows
        from irreducibility of this module. Write $v_j$
        for the $j^{\rm th}$ component of a vector
        in $\vsp{n}{q}$. \\

        To begin with, consider an arbitrary subset
        $B = \{ w_1, w_2, \ldots , w_m \}$ of
        $N \backslash \{ 0 \}$. The subgroup of $G_0 = \sym{n}$
        that fixes $B$ pointwise consists of
        precisely those permutations that interchange coordinates $j$ of
        $\vsp{n}{q}$ for which the sequences $\{ (w_i)_j \}_{i=1}^m$
        are equal, so $B$ is a base if and only if all such sequences
        are distinct. Hence $b(G_0) \ge \lceil \log_q { n } \rceil
            = k$; conversely, setting $m = k$ and
        $(w_i)_j$ equal to an element of $\mathbb{F}_q$ corresponding
        to the $j^{\rm th}$ digit of $i$ in a base $q$ representation, we
        see that the relevant sequences are distinct and that
        $\sum_{j=1}^n (w_i)_j = q^{k-1}
            \sum_{x \in \mathbb{F}_q} x - 0 = 0$, ensuring
        $b(G_0) = k$. Note also that no $w_i$ in a base of $k$ elements
        can have more than $q^{k-1}$ components $(w_i)_j$ equal
        to any given value of $\mathbb{F}_q$, lest the above argument
        yield $b(G_0) \ge 1 + \lceil \log_q{(q^{k-1} + 1)} \rceil + 1 =
                k + 1$. In conjunction with the fact that
        the $n = q^k - 1$ coordinates in any element of $\vsp{n-1}{q}$
        are required to add to 0, this proves such $w_i$ have
        \textsl{exactly} $q^{k-1}$ components equal to each non-zero
        element of $\mathbb{F}_q$. \\

        Now, each base $B \subseteq \vsp{n}{q}$ of $G$ gives rise to a base
        $(B \backslash \{ v \}) - v$ of $G_0$ for any $v \in B$,
        while bases $B$ of $G_0$ give rise to bases $B \cup \{ 0 \}$ of
        $G$. It follows that $b(G) = b(G_0) + 1 = k + 1$ as claimed. Moreover,
        the vertices $v = ( 1, -1, 0, 0, \ldots , 0 )$, $-v$, and
        $0 = 0v$ of $G$'s Saxl hypergraph $\saxh$ cannot have
        a common neighbour: were such a neighbour $w$ available,
        the non-zero vectors among $v - w$, $-v - w$, and $-w$ would
        all belong to bases of minimal size for $G_0$ or be 0, which in light of
        the above would mean that all $x \in \mathbb{F}_q$ appear
        as a component the same number of times $q^{k-1} - \ind{x = 0}$
        in each. But, if none of the three vectors is 0, that would
        entail the impossibility $\{ 1 - w_0, -1 - w_1 \} =
        \{ -1 - w_0, 1 - w_1 \} = \{ -w_0, -w_1 \}$; if one of
        them is 0, by contrast, the $(q - 1)q^{k-1}$ non-zero components
        of the each remaining vector would lie among the 2 components
        in which it differs, forcing $q = 3$ and $k = 1$ for the contradiction
        that 3 points in a hypergraph of $| \Omega | = q^{n-1} =
            3^{3^1 - 1 - 1} = 3$ vertices share a common neighbour.
        Hence $\gamma_3 = 0$.
    \end{proof}

    \begin{remark}
        In the case that $q$ is even and $k \ge 2$, the above proof can
        be adapted to show $\gamma_4 = 0$ by taking the vertices
        $v = (1, 1, 0, 0, \ldots, 0)$, $v' = (1, 0, 1, 0, 0, \ldots , 0)$,
        $v + v'$, and $0 = 0v$.
    \end{remark}

    \thm[g2]{gos-aff} together prove \thm{goss}.

    \section{\label{sec:valency}Valency}
    Let us now consider the degrees of vertices in $\saxh$.
    Vertex-transitivity of $\saxh$ (cf.\ \lem{lem21}) ensures that it
    is a regular hypergraph in the particular case that $G$ is transitive,
    but the following degree formula holds in general.

    \begin{lemma}
        \label{lem:valency}
        Let $G \le \Sym{(\Omega)}$ be a finite permutation group.
        The degree of $\alpha \in \Omega$ in $\saxh$ is
        \[ d = \frac{| \mathcal{B} |}{(b(G) - 1)!} =
            \frac{n \left| G_\alpha \right|}{( b(G) - 1 )!}, \]
        where $n$ is the number of orbits of $G_\alpha$ on its set
        $\mathcal{B}$ of ordered bases of size $b(G) - 1$.
        In particular, $\saxh$ is regular for transitive $G$.
    \end{lemma}
    \begin{proof}
        We count the number of ordered bases for $G$ of size $b(G)$ that
        begin with $\alpha$ in two ways. Firstly, each such base arises
        from an ordering of an edge $E \owns \alpha$ of $\saxh$ in which
        $\alpha$ comes first. There are $d$ such edges and $(b(G) - 1)!$
        such orderings of each. Secondly, each such base consists of
        $\alpha$ followed by an ordered base of $b(G) - 1$ elements
        for $G_\alpha$. There are
        $| \mathcal{B} |$ such bases, and $G_\alpha$ of course acts
        semiregularly on them with $n$ orbits.. Thus
        \[ d (b(G) - 1)! = | \mathcal{B} | = n |G_\alpha| \]
        as sought.
    \end{proof}

    The non-trivial denominator $(b(G) - 1)!$ in \lem{valency} makes questions
    around primality and parity of the valency more difficult to resolve
    than for $b(G) = 2$. Even more unfortunately, the number of  permutations of an ordered base $B$ containing $\alpha$ lying in the same $G_\alpha$-orbit as $B$ isn't necessarily consistent across choices of bases. 
    We thus content ourselves with a discussion of $G$ for which
    $b(G) \in \{ 3, 4 \}$ is ``small,'' beginning with a classification
    of those for which $\saxh$ has prime valency:

    \begin{theorem}
        \label{thm:prime-val}
        Suppose that $G \le \Sym{(\Omega)}$ is a transitive group of
        base size $3$ or $4$, and that $\saxh$ has prime valency. Then
        $(G, |\Omega|) = (\sym{4}, 4)$.
    \end{theorem}
    \begin{proof}
        We begin by proving the simple result that, if the order of
        $H \le \Sym{(\Omega)}$ has $n \le b(H)$ prime factors (counting
        repetitions), any $\alpha \in \Omega$ not fixed by $H$ lies
        in an $H$-orbit of prime order, and also a base of minimal size $n = b(H)$
        for $H$. Indeed, this is clear for $n = 0$. Moreover, if it holds for all $n$
        less than or equal to some fixed $k$,
        then the stabiliser of any non-fixed point $\alpha \in \Omega$
        in a group $H$ for which $|H|$ has $k + 1 \le b(H)$ prime factors will
        itself have an order $| H_\alpha | = | H |/| \alpha^H |$
        with $k + 1 - m \le k \le b(H) - 1 \le b(H_\alpha)$ prime factors,
        where $m \ge 1$ is the number of prime factors of $| \alpha^H |$.
        It follows that equality holds throughout, proving that
        $| \alpha^H |$ is prime, that $b(H) = k + 1$, and (since
        $b(H_\alpha) = b(H) - 1$) that $\alpha$ lies in a base for $H$ of
        minimal size. Strong induction yields the overall claim. \\

        Now, the left-hand side of the valency formula
        \[ d (b(G) - 1)! = n |G_\alpha| \]
        has exactly $b(G_\alpha) = b(G) - 1$ prime factors if
        $d$ is prime and $b(G) \le 4$. The fact that $|G_\alpha|$ divides
        the right-hand side thus allows the application of the result in
        the preceding paragraph; as a further consequence of
        the divisibility relation, $| G_\alpha | = d (b(G) - 1)!$, and $G_\alpha$
        has $n = 1$ orbits on its bases of minimal size. This single orbit gives
        rise to a single orbit on all points of $\Omega$ that appear
        in such bases, which the applied result moreover shows is
        a transitive orbit $\Delta$ of prime size $p$ on all points not fixed
        by $G_\alpha$. Note that $\Delta' = \Omega \backslash \Delta$
        consists of the fixed points of $G_\alpha$, a block of imprimitivity,
        so that $\Delta$ is a union of blocks. \\

        Since $| \Delta | = p$ is prime, we have either $| \Delta' | = p$
        or $| \Delta' | = 1$. In the former case, the transitivity of
        $G$ would guarantee that $G_{(\Delta)}$ is permutation isomorphic
        to $G_{(\Delta')} = G_\alpha$. It follows that there would be
        a two-point stabiliser in $G$ equal to $G_{(\Delta)} \cap G_{(\Delta')} =
            G_{(\Omega)} = 1$, contravening the requirement that
        $b(G) \in \{ 3, 4 \}$. The only possibility is therefore $| \Delta' | = 1$
        and $\Delta = \Omega \backslash \Delta' = \Omega \backslash \{ \alpha \}$,
        so that $| \Omega | = p + 1$ and $G$ is 2-transitive.

        Our task consequently reduces to a search through the list of
        2-transitive finite groups in \cite[\S7.7]{perm_grp}
        (cf.\ the proof of \thm{complete}). Groups with socles $\symp{2m}{2}$,
        $\psu{3}{q}$, $\suz{q}$, $\ree{q}$, and $\psl{n}{q}$ can largely be
        discounted on the basis that $| \Omega | - 1$ takes composite values
        $2^{m-1} (2^m \pm 1) - 1 = (2^{m-1} \pm 1) (2^m \mp 1)$, $q^3$, $q^2$,
        and $q^3$, respectively. The sole exceptions are $\symp{2}{2} \cong \sym{3}$
        and $\symp{4}{2} \cong \sym{6}$ in their actions on 3 and 6 points,
        neither of which has base size 3 or 4, and
        $\psl{n}{q}$ in its action on the projective line in the case 
        that $n = 2$ with $q$ prime. Primality of $q$ ensures
        $\psl{2}{q} \le G \le \pgl{2}{q}$, so that $b(G) = 2 + \ind{q \neq 2}$
        and point stabilisers in $G$ have order $q (q - 1)/(q - 1, 2)$
        (if $G = \psl{2}{q}$) or $q (q - 1)$ (if $G = \pgl{2}{q}$).
        To have $b(G) = 3$ and point stabiliser orders of the form
        $(b(G) - 1)! d = 2d$, where $d$ is prime, $G$ must be one of
        $\pgl{2}{3} \cong \sym{4}$ or $\psl{2}{5} \cong \alt{5}$.
        It can conversely be verified that the former, but not the latter,
        gives rise to $\saxh$ with prime valency via the result of
        \thm{complete} that $\saxh$ is complete. \\

        As for affine groups $G$ with socle $\vsp{n}{q}$, write $q = \pi^e$
        ($\pi$ prime). Primality of $| \Omega | - 1 = q - 1 = \pi^{en} - 1$
        shows that either $\pi = 3$ and $e = n = 1$ or $\pi = 2$ and
        $2^{en} - 1$, hence also $en$, is prime. The former scenario does not
        yield suitable groups $G$, as $G \le \vsp{1}{3^1} {:} \gammal{1}{3^1}
            \cong \sym{3}$ implies $b(G) \le 2$. The latter, by contrast, produces
        several candidates corresponding to distinct subcases. \\

        First suppose $n = 1$. Then $e$ is prime. It follows that
        the point stabilisers $G_0 \le \gammal{1}{q}$ of $G$ have order
        dividing the product of two primes $| \gammal{1}{q} | = (2^e - 1) e$;
        on the other hand, $| G_0 | = (b(G) - 1)! d \in \{ 2d, 6d \}$
        for prime $d$, so that $b(G) = 3$ and $\{ 2^e - 1, e \} = \{ 2, d \}$.
        Thus $G = \vsp{1}{2^2} {:} \gammal{1}{2^2} \cong \sym{4}$ acts
        on 4 points. \\

        We otherwise have $n \neq 1$. Note that $n$ is then prime and
        $e = 1$. The corresponding possibilities for $G_0$
        can be determined by sorting through the various classes of
        2-transitive affine groups. If
        $\lins{n}{2} \le G_0 \le \gammal{n}{2} =
            \lins{2}{2}$, the order of $G_0$ is divisible by $| \lins{n}{2} | =
                (2^n - 2^0)(2^n - 2^1) \ldots (2^n - 2^{n-1})$. This has
        more than three prime factors if $n > 2$, contradicting
        the earlier result $| G_0 | \in \{ 2d, 6d \}$, while $n = 2$
        yields $G = \vsp{2}{2} {:} \ling{2}{2} \cong \sym{4}$. As for when
        $G_0 \ge \symp{n}{2}$ or $G_0 \le \gammal{n}{2} = \gammal{6}{2}$,
        the consequence that $n$ is even again produces
        $G = \vsp{2}{2} {:} \symp{2}{2} \cong \sym{4}$. Finally,
        the sporadic examples are immediately ruled out on the basis that
        none have $q = 2$ and $n$ prime. \\

        To complete the proof, the symmetric, alternating, and
        unclassified 2-transitive groups can be filtered on the basis of
        the requirement that $| \Omega | - 1$ is prime and $b(G) \in \{ 3, 4 \}$.
        This leaves only the possibilities listed in the statement of the theorem
        and $(G, | \Omega |) \in \{ (\alt{6}, 6), (\psl{2}{11}, 12),
            (\pgl{2}{11}, 12) \}$ to consider; all have complete Saxl hypergraphs
        by \thm{complete}, so the relevant valencies are easily evaluated
        as $\binom{| \Omega | - 1}{b(G) - 1}$.
    \end{proof}

    Further complications arise when investigating Eulerian circuits in $\saxh$, as was done for Saxl graphs in \cite[\S3.2]{burngiu}.
    As it happens, even valency is not sufficient (nor even necessary!)
    for the existence of an Eulerian circuit in a hypergraph:
    the problem of finding such a circuit is in general NP-complete, although
    this may not apply in the restricted case of Saxl hypergraphs. What can be
    said is that even vertex degrees and an even number of vertices are instead necessary
    and sufficient for a \text{flag-spanning tour} (a walk which uses each flag
    $(v \in e, e)$ exactly once), as defined in Bahmanian and Shan \cite{bah_shan}.
    \thm{flag-span} partially resolves this more tractable problem for $\saxh$
    corresponding to a \textsl{primitive} group $G$
    with $b(G) \in \left\{ 3, 4 \right\}$. For clarity's sake, we treat some of
    the more difficult almost simple groups in a lemma preceding the main proof. \\

    Frequent use will be made of the following simple combinatoric result:

    \begin{lemma}
        \label{lem:magiclemma}
        Let $S$ be a set, $E$ a set of properties, and $S_T$ denote
        the set of elements in $S$ satisfying at least
        the properties $T \subseteq E$. If $k, n \in \mathbb{N}$ and
        $X \subseteq \mathcal{P} (E)$ are such that
        $|S_T| \equiv k + n \ind{T \in X} \textrm{ mod } 2n$
        for any $T$ --- including $T = \emptyset$ --- then the number of elements
        in $S$ which satisfy no property in $E$ is congruent
        to $k \ind{E = \emptyset} + n \ind{|X| \textrm{ odd}}$ modulo $2n$.
        In particular, for $X = \emptyset$, this is $k$ if $E = \emptyset$
        and 0 otherwise.
    \end{lemma}
    \begin{proof}
        By the Inclusion--Exclusion Principle, the number of elements
        in $S$ which satisfy no property of $T$ is

        \begin{align*}
            \sum_{T \subseteq E} (-1)^{|T|} | S_T | & \equiv
            \sum_{T \subseteq E} (-1)^{|T|} (k + n \ind{T \in X}) \\
            & = \sum_{T \subseteq E} (-1)^{|T|} k +
                \sum_{T \in X} (-1)^{|T|} n \\
            & \equiv \sum_{i=0}^{|E|} \binom{|E|}{i} (-1)^i k +
                \sum_{T \in X} n \\
            & = n|X| + \left\{ \begin{array}{ll}
                k, & |E| = 0, \\
                k (1-1)^{|E|} = 0, & \textrm{otherwise,}
            \end{array} \right. \\
            & \equiv k \ind{E = \emptyset} + n \ind{|X| \textrm{ odd}},
        \end{align*}

        where all congruences are modulo $2n$.
    \end{proof}

    \begin{lemma}
        \label{lem:flag-span-simp}
        Let $G \le \Sym{(\Omega)}$ be one of the almost simple groups
        in \tbl{nofstsimp}. Then $\saxh$ has odd valency if and only if
        $G = \pgammal{2}{2^e}$ has a subgroup of type $\ling{1}{q} \wr \sym{2}$
        for a point stabiliser, with $e$ a square-free odd integer, or
        $G \in \{ \psl{2}{q}, \pgl{2}{q} \}$ has a parabolic point stabiliser of
        type $P_1$, where $q \equiv 3 \textrm{ mod } 4$.
    \end{lemma}
    \begin{table}
        \caption{\label{tbl:nofstsimp}Almost simple groups
            $T \le G \le \Aut{(T)}$ acting with point stabiliser
            $H \ge H \cap T$ such that $b(G) \in \{ 3, 4 \}$ and
            $4 \nmid |H|$.}
        \begin{tabular}{lll}
        \toprule
            $T$ & $H \cap T$ & Comments \\ \midrule
            $\ree{q}$ & $\cyc{q^3} {:} \cyc{q-1}$ &
                $q = 3^{2n+1} \ge 27$ \\ 
            $\psl{2}{q}$ & $\cyc{q} {:} \cyc{(q-1)/2}$ &
                $q$ odd; type $P_1$ \\ 
            $\psl{2}{q}$ & $\dih{2(q-1)}$ & $q$ even, $G \neq T$; type
                $\ling{1}{q} \wr \sym{2}$ \\ 
            $\psl{2}{q}$ & $\dih{2(q+1)}$ & $q$ even; type
                $\ling{1}{q^2}$\\
            $\psl{3}{3}$ & $\cyc{13} {:} \cyc{3}$ & $G = \psl{3}{3}.2$;
                type $\ling{1}{q^3}$ \\
                \bottomrule
        \end{tabular}
    \end{table}
    \begin{proof}
            Computations show $\psl{3}{3}.2$ gives rise to valency 9750, so
            that it remains to consider the four infinite families.
            Note that, whenever projective groups over a field $\mathbb{F}_q$
            are discussed below, the points of the projective line will be
            identified in the natural way
            with $\mathbb{F}_q \cup \{ \infty \}$. \\

            If $\psl{2}{q} \le G \le \pgammal{2}{q}$
            (where $q$ is odd) has a point stabiliser of type $P_1$,
            then $\Omega$ may be identified with the projective line.
            Per \cite[Remark 9.2(ii)]{simpsoltab}, $b(G) = 3$ if $G$ is
            one of the groups in case~\ref{complete:lin} of \thm{complete} and
            $b(G)=4$ otherwise. In the former eventuality, \thm{complete} establishes that
            $\saxh$ is complete, and thus the claim in case~\ref{flag-span:lin1}
            that its valency $\binom{| \Omega | - 1}{2} = q(q-1)/2$ is odd if
            and only if $q \equiv 3 \textrm{ mod } 4$ (which, in particular,
            eliminates the possibility that $q$ is a perfect square).
            As for when $b(G) = 4$,
            first consider the set $\mathcal{B}$ of 2-element ordered bases
            for $G_{0, \infty}$. These have the form $( \alpha, \beta )$
            for certain $\alpha, \beta \in \mathbb{F}_q$.
            Since the pointwise stabiliser in $\pgammal{2}{q}$ of a triple
            $( 0, \infty, \alpha )$
            is identical to that of
            $( 0, \infty, -\alpha )$, the map
            $f : \, (\alpha, \beta) \mapsto (-\beta, \alpha)$
            permutes $\mathcal{B}$. Clearly $f^4 = 1$, while the odd order of
            $\mathbb{F}_q$ means $f^2$ has a unique fixed point
            $B_{0, 0} \notin \mathcal{B}$. Hence $| \mathcal{B} |$ is
            divisible by 4. The 2-transitivity of the action of $G$
            implies that there are $(|\Omega| - 1) | \mathcal{B} |$
            bases of $G_0$ with size $b(G) - 1 = 3$, in turn yielding
            even valency
            $(|\Omega| - 1) | \mathcal{B} |/(b(G) - 1)! = q | \mathcal{B} |/6$ of
            $\saxh$. \\

            For $G$ with socle a Ree group in the indicated
            2-transitive action, recall
            from the proof of \thm{complete} that $b(G) = 3$ and that
            there is a permutation representation on the set
            $\mathbb{F}_q^3 \cup \{ \infty \}$ (where $q = 3^{2n+1}$).
            The pointwise stabiliser of $\{ 0, \infty \}$
            in this representation consists of the permutations
             \[ n_\kappa : \, (\eta_1, \eta_2, \eta_3) \mapsto (
                \kappa \eta_1, \kappa^{\sqrt{3q} + 1} \eta_2,
                \kappa^{\sqrt{3q} + 2} \eta_3
            ), \qquad \infty \mapsto \infty \]
            (where $\kappa \in \mathbb{F}_q^{\times}$) composed
            with field automorphisms $f \in \Aut{(\mathbb{F}_q)} \cap G$.
            Bearing this in mind, the set $\mathcal{B} \subseteq
                \mathbb{F}_q^3 \cup \infty$ of $v = (\eta_1, \eta_2, \eta_3)$
            for which $( 0, \infty, v )$ is an ordered basis of $G$ divides
            into 3 classes:

            \begin{itemize}
                \item If $\eta_1 \neq 0$, then the stabiliser of
                    $\{ 0, \infty, v \}$ consists of
                    those $n_\kappa f$ for which the triple
                    $(\kappa f(\eta_1), \kappa^{\sqrt{3q} + 1} f(\eta_2),
                        \kappa^{\sqrt{3q} + 2} f(\eta_3))$ equals
                    $(\eta_1, \eta_2, \eta_3)$. Comparison of
                    the first coordinates shows $\kappa = \eta_1 f(\eta_1)^{-1}$,
                    which upon substitution into the second and third terms
                    produces the requirements that $\eta_2 \eta_1^{-(\sqrt{3q}+1)}
                        = f(\eta_2 \eta_1^{-(\sqrt{3q}+1)})$ and
                    $\eta_3 \eta_1^{-(\sqrt{3q}+2)}
                        = f(\eta_3 \eta_1^{-(\sqrt{3q}+2)})$. The number of
                    points in $\mathcal{B}$ corresponding to each $\eta_1$ is thus
                    equal to the number of pairs $(\alpha, \beta) \in
                        \mathbb{F}_q^2$ not fixed by any non-trivial
                    $f \in \Aut{(\mathbb{F}_q)} \cap G$. For any set $S$
                    of such field automorphisms (empty or not), however,
                    the pairs which \textsl{are} fixed
                    by the automorphisms in $S$ are precisely those
                    fixed by a generator of $\langle S \rangle \le
                        \Aut{(\mathbb{F}_q)} \cong \cyc{2n+1}$,
                    which amount to $| \mathbb{F}_Q |^2 \equiv 1
                        \textrm{ mod } 4$ pairs
                    for some $\mathbb{F}_Q \le \mathbb{F}_q$.
                    It follows from \lem{magiclemma} that
                    the number of $v$ of interest for each $\eta_1$
                    is congruent to 1 modulo 4 if $G$ contains
                    no non-trivial field automorphisms and 0 otherwise;
                    the number over all $\eta_1 \neq 0$ is thus congruent
                    to $(q - 1) \cdot 1 \equiv 2$ in the former situation,
                    and $(q - 1) \cdot 0 \equiv 0$ in the latter.
                \item A similar argument applies when $\eta_1 = 0$,
                    but $\eta_3 \neq 0$.  The first coordinates of the triples
                    $(\kappa f(0), \kappa^{\sqrt{3q} + 1} f(\eta_2),
                        \kappa^{\sqrt{3q} + 2} f(\eta_3))$ and
                    $(0, \eta_2, \eta_3)$ are equal for any $n_\kappa f$,
                    while the last two coordinates are equal if and only if
                    $\kappa = \kappa^{-(\sqrt{3q} + 2)(\sqrt{3q} - 2)} =
                    (\eta_3 f(\eta_3)^{-1})^{(2 - \sqrt{3q})}$ and
                    $\eta_2 \eta_3^M = f(\eta_2 \eta_3^M)$, where
                    $M = (\sqrt{3q} - 2)(\sqrt{3q} + 1)$ is fixed.
                    For each $\eta_3$, the choices of $\eta_2$ which give
                    $v \in \mathcal{B}$ are thus equinumerous with
                    the $\alpha \in \mathbb{F}_q$ not fixed by any non-trivial
                    $f \in \Aut{(\mathbb{F}_q)} \cap G$. Any set $S$
                    of field automorphisms fixes $| \mathbb{F}_Q | \equiv 3
                        \textrm{ mod } 4$ such $\alpha$ for some
                    $\mathbb{F}_Q \le \mathbb{F}_q$ (noting that
                    $Q^n = q \equiv 3 \textrm{ mod } 4$ for some $n$),
                    so that an application of \lem{magiclemma} and
                    enumeration over all $\eta_3 \neq 0$ yields
                    $(q - 1) \cdot 3 \equiv 2 \textrm{ mod } 4$
                    elements of $\mathcal{B}$ if $G$ contains
                    no non-trivial field automorphisms, and
                    $0 \textrm{ mod } 4$ otherwise.
                \item Finally, if $\eta_1 = \eta_3 = 0$,
                    the fact that $\sqrt{3q} + 1$ is even
                    ensures $n_{-1}$ fixes $v$. This case thus makes
                    no contribution.
            \end{itemize}

            Summing over all three cases reveals that $| \mathcal{B} |$ is
            divisible by 4, which in a similar manner to above implies
            the valency $(| \Omega |- 1) | \mathcal{B} |/(b(G) - 1)! =
                q^3 | \mathcal{B} |/2$ of $\saxh$ is even. \\

            The final two classes of groups in \tbl{nofstsimp}, which have
            base size $b(G) = 3$, will be addressed with the assistance
            of a slightly generalised argument, which we will apply in each case.
            
            \begin{clam}
            \label{clam_1}
                Let $H = \langle \tilde{H}, \mathcal{A} \rangle$, where
            $\mathcal{A} \le \Aut{(\mathbb{F}_q)}$ for some even $q$ and
            $\tilde{H} \le \pgl{2}{q}$ contains (the projection of)
            \[ \iota = \left( \begin{array}{ll} 0 & 1 \\ 1 & 0 \end{array} \right). \]
            Suppose that $H$ acts on a set $S$ of pairs of projective points
            such that $\tilde{\alpha} = \{ 0, \infty \} \in S$ and
            $\{ 0, \beta \} \in S$ is equivalent to $\{ \beta, \infty \} \in S$
            for any $\beta \notin \tilde{\alpha}$. If $b(H_{\tilde{\alpha}}) = 2$,
            then the set of 2-element ordered bases $\mathcal{B}$
            for $H_{\tilde{\alpha}}$ has $|\mathcal{B}| \equiv 0 \textrm{ mod } 4$ if and
            only if  either 
            \begin{enumerate}
                \item $| \mathcal{A} |$ is even,  or
                \item the number of ordered pairs
            $(\alpha, \beta) \in
                S \cap (\mathbb{F}_q^{\times})^2$, with $\beta \neq \alpha^{\pm1}$, not fixed
            by any non-trivial $f \in \mathcal{A}$ is a multiple of 8.
            \end{enumerate}  
            \end{clam}
            \begin{proof}[Proof of claim]
            Certainly, if  $|\mathcal{A}|$ is even, then
            $H_{\tilde{\alpha}}$ has a subgroup $\langle \iota, f_2 \rangle
                \cong \cyc{2}^2$ containing an involutory automorphism $f_2$,
            so that $| H_{\tilde{\alpha}} |$ and hence (by its semiregular action)
            $|\mathcal{B}|$ are divisible by 4; we may thus assume
            $|\mathcal{A}|$ is odd. \\

            The elements of $\mathcal{B}$ then fall into two disjoint classes. One consists
            of those bases $(\tilde{\beta}, \tilde{\gamma})$ such that
            $\tilde{\alpha} \cap \tilde{\beta}, \tilde{\alpha} \cap \tilde{\gamma} \neq \emptyset$.
            There are then unique expressions $\{ t_0, \ell \}$ and $\{ t_1, m \}$
            for $\tilde{\beta}, \tilde{\gamma}$ in which $t_0, t_1 \in \tilde{\alpha}$ and
            $\ell, m \in \mathbb{F}_q^{\times}$, lest $\tilde{\alpha}$ appear
            in an irredundant base for its stabiliser. It follows that the stabiliser of
            $(\tilde{\beta}, \tilde{\gamma})$ in $H_{\tilde{\alpha}}$ is the stabiliser of
            the individual projective points $0$, $\infty$, $\ell$, and $m$, so that
            there are either 0 or $2^2 = 4$ bases of this form in $\mathcal{B}$
            for any pair of projective points $\ell, m$ (corresponding to the choices of
            $t_0, t_1 \in \{ 0, \infty \}$). The cardinality of this first class is
            thus a multiple of 4. \\

            As for the remaining bases in $\mathcal{B}$, consider
            the permutations induced by interchanging the elements in a base
            and by applying the transformation $\imath \in H_{\tilde{\alpha}}$
            to projective points. These actions are commuting involutions,
            and clearly leave the set of $(\tilde{\beta}, \tilde{\gamma}) \in \mathcal{B}$
            for which $\tilde{\alpha} \cap \tilde{\beta}, \tilde{\alpha} \cap
                \tilde{\gamma} \neq \emptyset$
            does \textsl{not} hold invariant. The fact that the elements of
            $\mathcal{B}$ are irredundant bases further ensures both actions
            are fixed-point free, guaranteeing that the group they generate
            has orbits of size $2^2 = 4$ on all bases except those fixed
            under the composition of the two actions; namely, bases of the form
            $( \{ \alpha, \beta \}, \{ \alpha^{-1}, \beta^{-1} \} )$
            for some $\alpha, \beta \in \mathbb{F}_q^{\times}$, where the second pair
            is the image of the first under $\iota$. \\

            Now, such a base clearly cannot have $\beta = \alpha$ (singletons),
            $\beta = \alpha^{-1}$ (identical pairs), or $\{ \alpha, \beta \}$
            fixed pointwise by any non-trivial $f \in \mathcal{A}$. Conversely,
            if none of these three conditions hold, then it will be shown that
            no $1\neq h \in H_{\tilde{\alpha}}$ leaves
            $( \{ \alpha, \beta \}, \{ \alpha^{-1}, \beta^{-1} \} )$
            invariant, which is therefore an irredundant base in light of
            the fact that
            $b(H_{\alpha}) = 2$. Indeed, any $h \in H_{\alpha}$ can be written
            as $h = \iota^j M_x f$ for suitable $j \in \{ 0, 1 \}$,
            $M_x = \left( \begin{array}{ll} 1 & 0 \\ 0 & x \end{array} \right) \in
                \tilde{H}$, and $f \in \mathcal{A}$.
            Then $h$ leaves the given unordered pairs invariant exactly when there exist
           $\pi_+, \pi_- \in \Sym{(\{\alpha, \beta\})}$ satisfying

            \begin{minipage}{0.45\textwidth}
                \begin{align}
                    \pi_+ (\alpha) & = h(\alpha) = (x f(\alpha))^{(-1)^j} \\
                    \pi_+ (\beta) & = h(\beta) = (x f(\beta))^{(-1)^j}
                \end{align}
            \end{minipage}
            \begin{minipage}{0.45\textwidth}
                \begin{align}
                    \pi_- (\alpha)^{-1} & = h(\alpha^{-1}) = (x f(\alpha^{-1}))^{(-1)^j} \\
                    \pi_- (\beta)^{-1} & = h(\beta^{-1}) = (x f(\beta^{-1}))^{(-1)^j}
                \end{align}
            \end{minipage}

            Taking the products of (1), (3) and (2), (4) establishes that 
            $\beta \alpha^{-1} = x^{2 (-1)^j} = \alpha \beta^{-1}
                = (\beta \alpha^{-1})^{-1}$ if $\pi_+ \neq \pi_-$
            and $x^{2 (-1)^j} = 1$ otherwise. Since $\mathbb{F}_q$
            has characteristic 2, however, these criteria imply
            $\beta = \beta \alpha^{-1} \alpha = \alpha$ or $x = 1$,
            respectively; the former possibility is prohibited
            by assumption. With the observation that
            $\pi_+ \in \Sym{(\{\alpha, \beta\})}$ has order dividing 2,
            it follows from (1), (2) that
            \begin{align*} \alpha = \pi_+^2 (\alpha) = (x f(\pi^+(\alpha)))^{(-1)^j} &=
                (x f((x f(\alpha))^{(-1)^j}))^{(-1)^j}\\
                &=
                f^2 (\alpha^{(-1)^j (-1)^j}) = f^2 (\alpha)
            \end{align*}

            and similarly for $\beta$. Since $(\alpha, \beta)$ is fixed
            by no non-trivial $f^2$ in the group $\mathcal{A}$ of odd order,
            we conclude $f = {\rm Id}$, which on substitution into (1), (2) becomes
            \[ \pi_+ (\alpha) = (x f(\alpha))^{(-1)^j} = \alpha^{(-1)^j} \qquad {\rm and}
            \qquad \pi_+ (\beta) = (x f(\beta))^{(-1)^j} = \beta^{(-1)^j} \]

            The possibility $\pi_+ \neq {\rm Id}$ is ruled out by the requirement that
            $\beta \notin \{ \alpha, \alpha^{-1} \}$, while $\pi_+ = {\rm Id}$ and $j = 1$
            yields $\alpha = \alpha^{-1}$, $\beta = \beta^{-1}$, and therefore
            another contradiction $\alpha = 1 = \beta$. We thus have
            $h = \iota^0 M_1 \cdot {\rm Id} = 1$ as claimed. Noting that
            $\alpha, \beta$ and $\beta, \alpha$ give rise to the same base
            in $\mathcal{B}$, it is in turn immediate that 4 divides $| \mathcal{B} |$
            exactly when 8 divides the number of pairs $(\alpha, \beta)$ which satisfy
            the stated properties. 
            \end{proof}
            For $\pgl{2}{q} < G \le \pgammal{2}{q}$ ($q = 2^e$ even)
            with a point stabiliser of type $\ling{1}{q} \wr \sym{2}$,
            one may identify $\Omega$ with the set of distinct pairs of
            projective points. Applying \claim{clam_1} with $H = G$ and
            $S = \Omega$ thus reduces the determination of whether
            the valency $| \mathcal{B} |/2$ is even to calculating
            the parity of $| G \cap \Aut{(\mathbb{F}_q)}|$ and
            the number of pairs of points $\alpha$ and
            $\beta \neq \alpha, \alpha^{-1} \in \mathbb{F}_q^{\times}$
            not fixed by a non-trivial field automorphism in $G$.
            The field elements $\alpha \in \mathbb{F}_q$ which
            \textsl{are} fixed by any
            $S \subseteq \Aut{(\mathbb{F}_q)} \backslash \{ {\rm Id} \}$,
            however, are precisely those in the field
            $\mathbb{F}_Q \le \mathbb{F}_q$ fixed pointwise
            by a generator of $\langle S \rangle$; there are thus
            $(Q - 1)(Q - 1 - 2) + 1 = (Q - 2)^2$ pairs of non-zero $\alpha$
            and $\beta \neq \alpha, \alpha^{-1}$ fixed by such $S$. (The added 1
            accounts for the case $\alpha = \alpha^{-1} = 1$.) Since $Q > 1$
            is a power of 2, this is congruent to 0 modulo 8 if $Q = 2$ and
            4 otherwise, so that by \lem{magiclemma} the number of pairs
            $(\alpha, \beta)$ is a multiple of 8 exactly when
            $(G \cap \Aut{(\mathbb{F}_q)}) \backslash \{ {\rm Id} \}$
            has an even number of subsets $S$ for which $Q = 2 = 2^1$. \\

            These subsets are, in fact, those which generate
            $\Aut{(\mathbb{F}_q)} \cong \cyc{e}$: the fixed field of
            $S \subseteq \Aut{(\mathbb{F}_q)}$, or (equivalently) of
            the group $\langle S \rangle$, has size
            $q^{1/|\langle S \rangle|} = 2^{e/|\langle S \rangle|}$.
            It follows that there
            are none unless $\Aut{(\mathbb{F}_q)} \le G$, in which case
            the generating subsets are precisely those not contained
            in the unique index-$p$ subgroup of $\Aut{(\mathbb{F}_q)}$
            for any prime divisor $p$ of $e$ (that is, any maximal subgroup).
            The intersection of the index $p_i, \ldots , p_j$ subgroups has
            $2^{e/(p_i \ldots p_j) - 1}$ subsets not containing $\rm Id$.
            The Inclusion-Exclusion Principle hence tallies
            \[ \sum_{H \subseteq \{ p \textrm{ prime} \, : \, p \mid E \}}
                2^{e/(\prod_{p \in H} p) - 1} (-1)^{|H|} \]
            subsets $S$ of interest, a sum composed entirely of even terms
            except in the case of a unique term for which $e = \prod_{p \in H} p$.
            Overall, $\saxh$ has even valency unless $G \ge \Aut{(\mathbb{F}_q)}$
            and
            $e$ is an odd product of distinct primes. This corresponds
            to the stated exception $G = \pgammal{2}{2^e}$
            for odd square-free $e$. \\

            Finally, consider the case where $\pgl{2}{q} \le G \le \pgammal{2}{q}$
            ($q$ again even) has point stabilisers of type $\ling{1}{q^2}$.
            We follow the proof of \cite[Lemma~4.8]{simpsoltab} in identifying
            a point stabiliser $G_0$ thereof with a group
            $\psu{2}{q} \le R \le \psigmau{2}{q}$ and $\Omega$ with the set
            \[ \{ \{ 0, \infty \} \} \cup
                \{ \{ \alpha, -\alpha^q = \alpha^q \} \, : \,
                \alpha \in \mathbb{F}_{q^2}^{\times}, \, \alpha^{q+1} \neq 1 \} \]
            of orthogonal pairs of non-degenerate 1-dimensional subspaces of
            the natural module $\vsp{2}{q^2}$ of $R$. (Observe $-1 = 1$
            in $\mathbb{F}_{q^2} \ge \mathbb{F}_2$.) Noting that this configurations
            satisfies $R_{\tilde{\alpha}} = R \cong G_0$, we may then take
            $H = R$ and $S = \Omega$ in \claim{clam_1} to conclude that the valency of
            $\saxh$ is even if and only if
            $\mathcal{A} = R \cap \Aut{(\mathbb{F}_{q^2})}$
            has even size or the number of $\alpha \in \mathbb{F}_q^{\times}$
            such that $\alpha$ and $\alpha^q \neq \alpha, \alpha^{-1}$ are
            not simultaneously fixed by any automorphism
            in $\mathcal{A} \backslash \{ {\rm Id} \}$ is a multiple of 8.
            Since any automorphism which fixes $\alpha$ also fixes $\alpha^q$,
            it suffices to assume $| \mathcal{A} |$ odd and test
            the $\alpha \in \mathbb{F}_{q^2}^{\times}$
            for which $\alpha^{q \pm 1} \neq 1$. \\

            As such, take arbitrary
            $S \leq \mathcal{A}$.
            The elements $\alpha \in \mathbb{F}_{q^2}$ fixed by $S$
            are precisely those in the subfield $\mathbb{F}_Q \le
                \mathbb{F}_{q^2}$ where $Q = q^{2/|S|}$; given
            that $S$ has odd order,
            $Q$ can in fact be written as $Q_0^2$ where  $Q_0=q^{1/|S|}$. There are thus $Q_0^2 - 1$
            values of $\alpha \in \mathbb{F}_{q^2}^{\times}$ fixed
            by $S$. Of these, $(Q_0^2 - 1, q - 1) = Q_0 - 1$
            are solutions of $\alpha^{q - 1} = 1$, $(Q_0^2 - 1, q + 1) = Q_0 + 1$
            are solutions of $\alpha^{q + 1} = 1$, and $(Q_0^2 - 1, q - 1, q + 1) = 1$
            satisfy both. It follows that there are
            \[ Q_0^2 - 1 - (Q_0 - 1) - (Q_0 + 1) + 1 = (Q_0 - 1)^2 - 1 \equiv 0 \textrm{ mod } 8 \]
            field elements $\alpha \in \mathbb{F}_q^{\times}$ fixed
            by $S$ which are not solutions of
            either equation. Divisibility by 8 of
            the number of $\alpha \in \mathbb{F}_{q^2}^{\times}$ for which
            $\alpha^{q \pm 1} \neq 1$ and \textsl{no} non-trivial automorphism
            in $\mathcal{A}$ fixes $\alpha$ is consequently implied
            by \lem{magiclemma} (applied to the non-identity elements of $\mathcal{A}$), yielding even valency of $\saxh$
            to complete the proof.

    \end{proof}
    We now conclude the proof of \thm{flag-span}.
    
    \begin{proof}[Proof of \thm{flag-span}]
            By \cite{bah_shan}, $\saxh$ has a flag-spanning tour if and
            only if both its number of vertices $|\Omega|$ and its valency
            $d$ are even. Cases where the former does not hold are
            accounted for by \ref{flag-span:odd}; it thus suffices
            to consider $\saxh$ with odd valency $d$. \\

            Now, $( b(G) - 1 )! d = n \left| G_\alpha \right|$
            is not divisible by 4 for odd $d$ and $b(G) \le 4$.
            The point stabilisers in this scenario 
            must therefore be soluble by the Odd Order Theorem.
            On the other hand, the O'Nan--Scott Theorem classifies
            primitive permutation groups into just 5 classes.
            The point stabilisers in those of twisted wreath and diagonal
            type all have non-abelian composition factors, limiting
            the possibilities for soluble point stabilisers to groups of
            product, almost simple, and affine type. In the latter case,
            the fact that the elementary abelian socle $N$ has even order
            $|N| = |\Omega|$ and the point stabiliser (i.e.\ factor group
            modulo $N$) is soluble ensures case \ref{flag-span:solv} applies.
            This leaves only the former two types to consider. \\

            In the case of a group $G$ of product type, let
            $L$, $J$, and $P \le \sym{k}$ denote the base group,
            point stabiliser of the base group, and top group,
            respectively. Write $T^n$ for the socle of $L$.
            Any point stabiliser in $G$ then has a subgroup isomorphic
            to $(J \cap T^n)^k$ and epimorphic image $P$, so that $\saxh$ has
            a flag-spanning tour if $|J \cap T^n|$ is even or $|K|$
            divisible by 4. This yields the restrictions on $K$
            in~\ref{flag-span:prod}, along with the requirement that $L$ be
            almost simple --- $L$ of diagonal type would have $J \cap T^n$
            containing a non-soluble factor $T$. Moreover, since
            $|J \cap T^n|$ is odd, it follows from the Odd Order Theorem and
            Schreier Conjecture that $J \unrhd J \cap T^n$ is soluble. \\

            The possibilities for $L$ and almost simple $G$ without
            a flag-spanning tour can now be
            obtained by simple enumeration. Burness
            \cite[Tables 9.4--7]{simpsoltab} tabulates
            the soluble maximal subgroups of almost simple groups
            with base size at least 3, while Burness and Huang
            \cite[Table~7.4]{prodsoltab} list the almost simple $L$
            with soluble maximal subgroups such that $b(L) < 3$,
            but $b(L \wr P) = 3$ for some $P$. 
            We undertake a case-by-case analysis of
            the orders of these maximal subgroups, referring to Kleidman
            and Liebeck \cite{k+l} and Bray \etal\cite{bhrd} as appropriate
            to determine the exact structures of subgroups of
            classical groups, in order to reduce the possibilities
            to those in \tbl[nofstbas]{nofstsimp}.
            The cases in \tbl{nofstsimp} are subsequently reduced
            to cases~\ref{flag-span:lin1} and~\ref{flag-span:lin2} above
            via \lem{flag-span-simp}. The result follows.
    \end{proof}

    \begin{remark}
        \label{rmk:fstka}
        Note that the possibility of a flag-spanning tour is left open
        in cases~\ref{flag-span:solv} and~\ref{flag-span:prod} of
        \thm{flag-span}. In general, the need to enumerate bases makes
        the task of proving odd valency significantly harder
        than that of proving even valency; we record only a few remarks
        in that direction here.
        \begin{itemize}
            \item All open possibilities have even degree.
            \item The two groups $\vsp{2}{6} {:} \dih{18}$ and
                $\vsp{2}{6} {:} (\cyc{7}{:}\cyc{6})$ both have
                base size 3 and satisfy
                criterion~\ref{flag-span:solv} above, yet
                only the Saxl hypergraph of the former has
                a flag-spanning tour. It remains unclear ---
                even in the case $b(G) = 2$ (see \cite{burngiu}) ---
                how to distinguish these possibilities
                for soluble groups.
            \item Burness and Huang \cite{prodsoltab} resolve
                the related question of Eulerian circuits
                for product type groups of the form $G = L \wr P$
                in the case $b(G) = 2$ (their Proposition~5.1).
                Much of their paper is equally relevant
                to the case $b(G) \in \{ 3, 4 \}$.
        \end{itemize}
    \end{remark}
    \bibliographystyle{abbrv}
    \bibliography{saxh.bib}
\end{document}